\documentclass[reqno]{amsart}

\usepackage{changes}

\usepackage[round,comma]{natbib}                %

\usepackage{amsmath, amsthm}
\usepackage{amsfonts,amssymb,dsfont}
\usepackage{nicefrac,mathrsfs}
\usepackage{bm}                                 %
\usepackage[backgroundcolor=white,bordercolor=orange]{todonotes}

\newcommand{\cA}{{\mathcal A}}  
\newcommand{\ccL}{{\mathscr L}}
\newcommand{\ccB}{{\mathscr B}}
\newcommand{\ccO}{{\mathscr O}}

\newcommand{\ccF}{{\mathscr F}}\newcommand{\cF}{{\mathcal F}}
\newcommand{\cG}{{\mathcal G}}
\newcommand{\cH}{{\mathcal H}}

\newcommand{\cQ}{{\mathcal Q}}
\newcommand{\cK}{{\mathcal K}}

\newcommand{\cX}{{\mathcal X}}

\newcommand{\cT}{{\mathcal T}}

\newcommand{\Ind}{{\mathds 1}}
\newcommand{\ind}[1]{\Ind_{\{#1\}}}

\newcommand{\R}{\mathbb{R}}

\newcommand{\bbF}{\mathbb{F}}

\newtheorem{theorem}{Theorem}[section]
\newtheorem{lemma}[theorem]{Lemma}              %
\newtheorem{proposition}[theorem]{Proposition}  %
\newtheorem{remark}[theorem]{Remark}  %
\newtheorem{assumption}[theorem]{Assumption}

\theoremstyle{definition}
\newtheorem{example}[theorem]{Example}
\newtheorem{definition}[theorem]{Definition}

\usepackage{enumerate,overpic,booktabs}

\usepackage{bbm}
\usepackage{stmaryrd}

\usepackage[colorlinks,urlcolor=red,citecolor=blue,linkcolor=red]{hyperref}

\newcommand{\fP}{\mathfrak{P}}
\DeclareMathOperator{\sem}{sem}

\DeclareMathOperator{\ac}{ac} 

\usepackage{geometry}
\definecolor{darkgreen}{rgb}{0,0.6,0}
\definecolor{col2}{rgb}{0,0,0.6}
\definecolor{col3}{rgb}{.50,0,98}
\definecolor{col4}{rgb}{.60,0,0}
\makeatletter
\@ifpackagelater{changes}{2012/01/11}{
  \definechangesauthor[color=green]{Green}
  \definechangesauthor[name={Thorsten Schmidt}, color=darkgreen]{ts}
 }{
  \definechangesauthor[Thorsten Schmidt]{ts}{darkgreen}  %
   \setauthormarkup[right]{}                       %
 }
 \makeatother

\begin{document}

\title{Affine processes under parameter uncertainty}

\author{Tolulope Fadina}
   \address{Freiburg University, www.stochastik.uni-freiburg.de}\email{tolulope.fadina@stochastik.uni-freiburg.de}
\author{Ariel Neufeld}
   \address{Nanyang Technological University, Division of Mathematical Sciences, Singapore} \email{ariel.neufeld@ntu.edu.sg}
   \author {Thorsten Schmidt}
\address{Freiburg  Institute of Advanced Studies (FRIAS), Germany. 
 University of Strasbourg Institute for Advanced Study (USIAS), France. 
 Department of Mathematical Stochastics, University of Freiburg, Eckerstr.1, 79104 Freiburg, Germany.}
 \email{thorsten.schmidt@stochastik.uni-freiburg.de}
 \date{\today}

\thanks{\textbf{Acknowledgements:} We want to thank  Tom Bielecki, Sam Cohen, Benedikt Geuchen, Mete Soner, Josef Teichmann, the participants of the ICERM workshop on Robust Methods in Probability and Finance,  the participants of the FWZ seminar and two anonymous referees for helpful comments. The numerical experiments were implemented by Jan Blechschmidt, TU Chemnitz. Financial support by Carl-Zeiss-Stiftung and the DFG, as well as ETH RiskLab  and the NAP Grant from NTU is gratefully acknowledged.  We also thank the 
Freiburg  Institute of Advanced Studies (FRIAS) for its hospitality and financial support.}

\maketitle

\begin{abstract}
We develop a one-dimensional  notion of  affine processes under parameter uncertainty, which we call \emph{non-linear affine processes}. This is done as follows:  given a set $\Theta$ of parameters for the process, we construct a corresponding non-linear expectation on the path space of continuous processes. By a general dynamic programming principle we link this non-linear expectation to a variational form of the Kolmogorov equation, where the generator of a single affine process is replaced by the supremum over all corresponding generators of affine processes with parameters in $\Theta$. This non-linear affine process yields a tractable model for Knightian uncertainty, especially for modelling interest rates under ambiguity. 

We then develop an appropriate It\^o-formula, the respective term-structure equations and study the non-linear versions of the Vasi\v cek and the Cox-Ingersoll-Ross (CIR) model. Thereafter we introduce the non-linear Vasi\v cek-CIR model. 
This model is particularly suitable for modelling interest rates when one does not want to restrict the state space a priori and hence the approach solves this modelling issue arising with negative interest rates.

	
\end{abstract}

\newcommand {\PP}{\mathbb P}
\newcommand {\QQ}{\mathbb Q}

\newcommand {\half}{\frac{1}{2}}

\section{Introduction} 

The modelling of a dynamic and unpredictable phenomenon like stock markets or interest rate markets 
is often approached via chosing an appropriate stochastic model. 
In many cases, the choice of the model is a delicate and difficult question. 
In complex dynamic environments like financial markets it is rather the rule than the exception that
unforeseen events lead to difficulties with the a-priori chosen model and improvements of the model 
have to be developed and implemented.

A promiment example in this direction is the role of affine short-rate models in the last 20 years: while around 2000, the property of the Vasi\v cek model that interest rates can become negative was heavily critizied and the non-negative Cox-Ingersoll-Ross (CIR) model was preferred, the consequences of the financial crises 2007-2008 leading to negative interest rates in the Euro zone rendered the CIR model no longer applicable and led to a renaissance of the Vasi\v cek model. 

This example illustrates the important question of \emph{model uncertainty}, which is  one of the most important topics in applied sciences and in particular plays a prominent role in finance, not only since the financial crisis.
The apparent risk of losses due to model mis-specification, called \emph{model risk} fostered the development of strategies which are \emph{robust} against model risk, typically leading to non-linear pricing rules. 
These robust strategies play a prominent role in the literature, see \cite{Martini,cont-06,eberlein2014bid,madan2016benchmarking,AcciaioBeiglbockPenknerSchachermayer.16,muhle2018risk,bielecki2018arbitrage} and the book \cite{guyon2013nonlinear}, to name just a few references in this direction.

A key observation in these works is that the single probability measure used in the classical approaches to specify a model has to be replaced by a family of probability measures (i.e.~a full class of models). 
Such an approach is very natural from the statistical viewpoint: when a model has certain parameters to be estimated, the estimators carry statistical uncertainty and  one considers confidence intervals instead, corresponding to a family of probability measures. The latter formulation of model risk is typically referred to as \emph{parameter uncertainty}, see \cite{AvellanedaLevyParas.95,wilmott1998uncertain,FouqueRen2014}, and is a major motivation for our research.

Examples in this direction are the notions of $g$-Brownian motion and $G$-Brownian motion referring to a Brownian motion with drift or volatility uncertainty, see \cite{Peng.97,D-Peng,PengG} and references therein. Most recently, this theory has been extended to more general approaches, so-called non-linear L\'evy processes, see \cite{NeufeldNutz2017} and \cite{DenkKupperNendel2017} in this regard.

Here we generalize this notion to affine processes under parameter uncertainty (called non-linear affine processes).
While a classical affine process corresponds to a single semimartingale law, we represent 
the  affine process under parameter uncertainty by a \emph{family} of  semimartingale laws whose differential characteristics are bounded from above and below by affine functions of the current states. Nonlinear L\'evy processes constitute the special case where the bounds do not depend on the state of the process. 
It seems important to stretch that for affine processes the bounds on drift and volatility are allowed to depend on the state of the process (in an affine way, however).
On the contrary, this state dependence leads to a number of additional difficulties and we therefore restrict ourselves to the simplest case, namely the one-dimensional case without jumps. 

It is our aim to provide the appropriate tools for incorporating parameter uncertainty in the prominent class of affine models. This naturally leads to a non-linear version of affine processes and associated non-linear expectations. 
After having established a dynamic programming principle, we establish the connection to the non-linear Kolmogorov equation.  
This allows us to study a number of interesting further steps, a non-linear version of the It\^o-formula, and non-linear affine term structure equations.

We also provide a number of examples: besides a non-linear variant of the Black-Scholes model and non-linear Vasi\v cek and Cox-Ingersoll-Ross (CIR) models we also introduce a non-linear Vasi\v cek-CIR model. In the latter model one can incorporate negative interest rates in combination with a CIR-like behaviour, solving the problem raised in many practical applications when the state space needed to be restricted to positive interest rates (see \cite{Carver2012}).

The paper is organized as follows: in Section \ref{sec:setup} we introduce non-linear affine processes. In Section \ref{sec:DPP-PDE} we prove that a dynamic programming principle holds. Section \ref{sec:Kolmogorov} provides the non-linear Kolmogorov equation and as examples the non-linear Vasi\v cek model and the non-linear CIR model. In Section \ref{sec:Ito} we provide a non-linear It\^o-formula together with some examples and in Section \ref{sec.:ATSM} we study (non-linear) affine term structure models. Section \ref{sec:modelrisk} studies the application to model risk and Section \ref{sec:conclusion} concludes.

\section{Setup}\label{sec:setup}
We begin with a short review of continuous affine processes in one dimension. For a detailed exposition we refer to \cite{DuffieFilipovicSchachermayer} and \cite{Filipovic2009}. Consider the canonical state space, which is either $\mathcal{X} = \mathbb{R}$ or $\cX=\R_{>0}$. 
A (time-homogeneous) Markov process $X$ with values in the state space $\cX$ is called \emph{affine} if the conditional
characteristic function of $X$ is exponential affine. This means that there exist $\mathbb{C}$-valued functions $\phi(t,u)$ and $\psi(t,u)$, respectively, such that 
$$\mathbb{E}[e ^{uX_T}\mid X_t] = e^{\phi (T-t,u)+\psi(T-t,u)X_t }$$
for all complex $u \in  \{ ix: x \in \mathbb{R}\}$, $ 0 \le t \le T$. 
The key for our non-linear formulation will be a characterization of $X$ in terms of  stochastic differential equations: more precisely, the affine process $X$ is the unique strong solution of the stochastic differential equation 
\begin{equation}
\label{SDE}
dX_t = (b^{0}+b^{1} X_t)dt + \sqrt{a^{0}+a^{1} X_t} dW_t, \qquad X_0=x
\end{equation}
where the drift parameter $b^{0}+b^{1} X_t$ and the diffusion parameter $a^{0}+a^{1} X_t$ depend on the current value of $X$ in an affine way.  Here, the process $W$ is a standard Brownian motion. It should be noted that, depending on the state space, not all parameter combinations are possible, but only those combinations which are admissible in the sense made precise in Theorem 10.2 in \cite{Filipovic2009}. For our case this implies that if on the one side $\cX=\R$ we necessarily have $a^1=0$ and $a^0 > 0$ and, on the other side, if $\cX=\R_{>0}$ we obtain $a^0=0$, $a^1>0$ and $b^0>0$. 
In addition, the coefficients $\phi$ and $\psi$ solve ODEs (classified as Riccati equations) which is the essence for the high degree of  tractability of  affine processes in the sense that explicit calculations are possible or efficient numerical methods are obtainable;  see \cite{DuffieFilipovicSchachermayer,Filipovic2009} for details and applications in this regard.

\subsection{Non-linear affine processes}
In this section we introduce the necessary tools for defining affine processes under parameter uncertainty. To this end, fix a final time horizon $T>0$ and let $\Omega =  C([0,T])$
 be the canonical space of continuous, one-dimensional paths. 
 We endow $\Omega$ with the topology of uniform convergence and denote by $\ccF$ its Borel $\sigma$-field. 
Let $X$ be the canonical process $X_t(\omega) = \omega_t$, 
 and let  $\mathbb{F}= (\ccF_t )_{t \geq 0}$ with $\ccF_t = \sigma (X_s, 0 \leq s \leq t)$ be the (raw) filtration generated by $X$.

As we are interested in semimartingale laws on $\Omega$ we begin by denoting by $\fP (\Omega)$  the Polish space of all probability measures on $\Omega$ equipped with the topology of weak convergence\footnote{The weak topology is the topology induced by the bounded continuous functions on $\Omega$. Then, $\fP(\Omega)$ is a separable metric space and we denote the associated Borel $\sigma$-field by $\ccB(\fP(\Omega))$. }. 
The process $X$ will be called a (continuous) $P$-$\bbF$-semimartingale, for $P \in \fP(\Omega)$, if there exist processes $B=B^P$ and $M=M^P$ such that $X=X_0+B+M$, where $B$ has continuous paths of (locally) finite variation $P$-a.s.,  $M$ is a continuous $P$-$\bbF$-local martingale and $B_{0}=M_0=0$.

It will be important in the following that, by Proposition (2.2) in \cite{NeufeldNutz2014}, $X$ is a $P$-$\bbF$-semimartingale if and only if it is a $P$-semimartingale with respect to the right-continuous filtration $\bbF_+=(\ccF_{t+})_{t \ge 0}$ or with respect to  the usual augmentation $\bbF_+^P$; here $\ccF_{t+}=\cap_{s > t}\ccF_s$. Hence, in the following we can consider semimartingales with respect to the raw filtration $\bbF$.

The $P$-$\bbF$-characteristics of a continuous semimartingale $X=X_0+B^P+M^P$ in the above representation is the pair $(B^P,C)$ where $C=\langle M^P \rangle$. The non-negative process $C$ does not depend on $P$, as the quadratic variation is a path property\footnote{This is because $C$ can be constructed as a single process not depending on $P$; that is, two measures under which $X$ has different diffusion are necessarily singular, see Proposition 6.6 in \cite{NeufeldNutz2014}, for the construction of $C$.}. 
For the following, we will focus on semimartingales where the semimartingale characteristics are absolutely continuous (a.c.), i.e.\ there exist predictable processes $\beta^P$ and $\alpha\ge 0$, such that 
	\begin{align}\label{eq:semimartchar}
		B^P &= \int_0^\cdot \beta^P_s ds, \quad
		C = \int_0^\cdot \alpha_s ds.
	\end{align}

A probability measure $P \in \fP(\Omega)$ is called a \emph{semimartingale law} for $X$, if $X$ is a $P$-$\bbF$-semimartingale. We denote by 
\begin{equation*}
\fP_{\sem}^{\ac}= \left\{  P \in \fP(\Omega) \,|\, X \text{ is a } P\text{-}\bbF\text{-semimartingale with a.c.\  characteristics}\right\}
\end{equation*} 
the set of all semimartingale laws of $X$ which have absolutely continuous characteristics. In the following we will always denote by $(\beta^P,\alpha)$ as in \eqref{eq:semimartchar} the differential characteristics of $X$ under $P\in\fP_{\sem}^{\ac}$.
 
Our main goal is to allow for a specific version of \emph{model risk} in the sense that there is uncertainty on the parameter vector $\theta=(b^{0},b^{1},a^{0},a^{1})$ of the affine process. %
We assume that there is additional information on bounds on the parameter vector $\theta$ and denote these finite bounds by  $\underline b^i$, $\bar b^i$, $\underline a^i$, $\bar a^i$, $i=0,1$, respectively. This leads to the compact set
\begin{equation}
\label{Theta}
\Theta = [\underline b^{0},\bar b^{0}] \times [\underline b^{1},\bar b^{1}] \times [\underline a^{0},\bar a^{0}] \times [\underline a^{1}, \bar a^{1}] \subset \R^2 \times \R_{\ge 0}^2. 
\end{equation}
We are interested in the intervals generated by the associated affine functions. In this regard, let
$B:=[\underline b^{0},\bar b^{0}] \times [\underline b^{1},\bar b^{1}]$ and $A:=[\underline a^{0},\bar a^0] \times [\underline a^{1},\bar a^{1}]$. Moreover, we denote for $a \in \R^2$,  $a:=(a^0,a^1)$ and similarly for $b\in \R^2$. Furthermore, let %
\begin{align} \label{def:AB}
	\begin{aligned} 
	b^*(x) & := \{ b^{0}+b^{1}x: b \in B \}, \\
	a^*(x) & := \{ a^{0}+a^{1}x^+: a \in A \}
\end{aligned} \end{align} 
for $x \in \R$ denote the associated set-valued functions.
As the state space will, in general, be $\R$ we have to ensure non-negativity of the quadratic variation  which is achieved using $(\cdot)^+:=\max\{\cdot,0\}$ in the definition of $a^*$.
Due to the nice structure of $\Theta$  the sets are always intervals: indeed,
	\begin{align} \label{eq:affineb*}\begin{aligned}
		b^*(x) &= [\underline b^{0}+ (\underline b^{1} \ind{x \ge 0} + \bar b^1 \ind{x <0}) x , 
		                              \bar b^{0}+ (\bar b^{1} \ind{x \ge 0} + \underline b^1 \ind{x <0}) x  ], \\
		a^*(x) &= [\underline a^{0}+\underline a^{1} x^+, \bar a^{0}+  \bar a^{1} x^+]. \end{aligned}
	\end{align} 

Clearly, it would be possible to consider more general $\Theta$, which is, however, not our focus here.
\begin{definition}
Let $\Theta$ be a set as in \eqref{Theta} with associated  $a^*$, $b^*$ as in \eqref{def:AB}, let $t \in [0,T]$, and let $P\in\fP_{\sem}^{\ac}$ be a semimartingale law. 
We call $P$ \emph{affine-dominated  on $(t,T]$} by $\Theta$, if $(\beta^P,\alpha)$ satisfy
\begin{align}\label{def:dominated}
     \beta^P_s \in b^*(X_s) , \quad \text { and }\quad \alpha_s \in a^*(X_s), 
\end{align}
for $dP\otimes dt$-almost all $(\omega,s) \in \Omega \times (t,T]$. If $t=0$, we call $P$ \emph{affine-dominated} by $\Theta$.
\end{definition}

An affine process can uniquely be characterized by its transition probabilities. We will make use of this fact for characterizing affine processes under model uncertainty. Moreover, we denote by $\ccO$ the considered state space, which will be either $\R$, $\R_{\ge 0}$ or $\R_{>0}$.

\begin{definition}\label{def:nlaffine}
Let $\Theta$ be a set as in \eqref{Theta} with associated $a^*$, $b^*$ as in \eqref{def:AB}.  
A \emph{non-linear affine process} starting at $x\in\ccO$  is the family of  semimartingale laws $P \in \fP_{\sem}^{\ac}$, such that
\begin{enumerate}[(i)]
  \item  $P(X_0=x)=1$, 
  \item  $P$ is affine-dominated by $\Theta$.
\end{enumerate}
\end{definition}
As explained in the introduction, parameter uncertainty is represented by a family of models replacing the one single model in the approaches without uncertainty: according to Definition \ref{def:nlaffine}, the affine process under parameter uncertainty is represented by a \emph{family} of semimartingale laws instead of a single one. 
We denote the semimartingale laws $P\in \fP_{\sem}^{\ac}$, satisfying $P(X_0=x)=1$ and being affine-dominated by $\Theta$  by $\cA(x,\Theta)$. Intuitively, this corresponds to a non-linear affine process starting in $x$.

It is well-known that the state space $\ccO$ needs to be chosen in correspondence with the choice of $\Theta$: indeed, the squared Bessel process is an affine process with state space $\R_{>0}$ (see, for example \cite{KaratzasShreve1988}, Prop. 3.22 of Ch. 3) and the set $\cA(x,\Theta)$ will be empty for $x <0$. 
To exclude additional difficulties in this direction, we call a family of non-linear affine processes $(\cA(x,\Theta))_{x \in \ccO}$ with state space $\ccO$ \emph{proper}, if either $a^0>0$ holds, or $\underline a^0=\bar a^0=0$ and $\underline b^0 \ge \nicefrac{\bar a^1} 2>0$. 

It is clear that in the case with $\ccO=\R$ the assumption $\underline a^0 > 0$ is sufficient for reaching the full state space. The case with non-negative state spaces is more delicate. We concentrate on the case $\ccO=\R_{>0}$. The following proposition gives a sufficient condition in this regard. It moreover shows that the non-linear affine process does not reach zero, in the sense that the event of reaching $0$ has zero probability under all $P \in \cA(x,\Theta)$.

\begin{proposition}\label{prop:positivity}
Let $x>0$,   and assume that $\underline a^0=\bar a^0=0$ and $\underline b^0 \ge \nicefrac{\bar a^1} 2>0$. %
 Then  for any $P\in \cA(x,\Theta)$ it holds that
$$ P(X_t >0, \ 0  \le t \le T) = 1. $$
\end{proposition}
\begin{proof}
Let $P \in \mathcal{A}(x,\Theta)$, denote by $\beta^P$ and $\alpha$ the associated processes from Equation \eqref{eq:semimartchar},
	and denote by $M^P$ the $P$-local martingale part of the $P$-semimartingale $X$. Moreover, for any $c \geq 0$ define
	\begin{equation*}
	\tau_c:=\inf\{t\geq 0\colon X_t=c\}.
	\end{equation*}
We need to show for any time $\cT>0$ that $P[\tau_0\leq\cT]=0$. To that end, fix an arbitrary $\cT>0$.
We adopt the method in \cite{gikhman2011Feller} to our setting:
let $\varepsilon$ such that $0<\varepsilon<x$.
Notice that  by continuity of the paths of $X$, we have that $X\geq\varepsilon>0$ on $[\![0,\tau_\varepsilon]\!]$. 
Moreover, 
by assumption,
\begin{align} m := \frac{2\underline b^0 - \bar a^1}{\bar a^1} > 0. \label{eq:m} \end{align}
It\^o's formula yields for any $t\geq 0$ that
\begin{align*}
(X_{\tau_\varepsilon \wedge t})^{-m} & = (X_0)^{-m} - \int_0^{\tau_\varepsilon \wedge t} m X_s^{-(m+1)} \beta^P_s \,ds 
+ \half \int_0^{\tau_\varepsilon \wedge t} m(m+1) X_s^{-(m+2)} \alpha_s  \,ds \\
& \quad 
-\int_0^{\tau_\varepsilon \wedge t} m X_s^{-(m+1)} \,dM^P_s
\end{align*}
Define the processes $M$ and $\tilde M$ by
\begin{align*}
M_t & =-\int_0^{t} m X_s^{-(m+1)} \,dM^P_s, \qquad t\geq 0,\\
\tilde M_t & =  \int_0^t   (X_{\tau_\varepsilon \wedge s})^{-(m+1)} dM^P_s,
\qquad t\geq 0.
\end{align*}
Clearly, both $M$ and $\tilde M$ are local martingales  and $M=-m\tilde M$ on $\llbracket 0,\tau_\varepsilon \rrbracket$. We now show that $\tilde M$ is a true martingale.
By the Burkholder-Davis-Gundy inequality, since $X\ge \varepsilon$ on $[\![0,\tau_\varepsilon]\!]$, we obtain that
\begin{align*}
E\bigg[\sup_{0 \le s \le T} |\tilde M_s|\bigg] 
&\le C E\bigg[ \bigg(\int_0^T X_{ \tau_\varepsilon\wedge s}^{-2(m+1)} \alpha_{ \tau_\varepsilon\wedge s} ds \bigg)^{\nicefrac 1 2 }\bigg] \\
&\le C E\bigg[ \bigg(\bar a^1  \int_0^T X_{ \tau_\varepsilon\wedge s}^{-2(m+1)+1} ds \bigg)^{\nicefrac 1 2 }\bigg] \\
&\le C (\bar a^1 \varepsilon^{-2(m+1)+1}T)^{\nicefrac 1 2}< \infty
\end{align*}
and hence $\tilde M$  is indeed a martingale. Therefore,  $M$ is  a true martingale on $\llbracket 0,\tau_\varepsilon \rrbracket$.
Next, since $P\in \cA(x,\Theta)$ and $m>0$, we obtain  the estimate
\begin{align*}
& X_{\tau_\varepsilon \wedge t}^{-m}\\ 
& \le  x^{-m} - \int_0^{\tau_\varepsilon \wedge t} m X_s^{-(m+1)} (\underline b^0 + \underline b^1 X_s) \, ds + \half \int_0^{\tau_\varepsilon \wedge t} m(m+1) X_s^{-(m+2)} \bar a^1 X_s \, ds + M_{\tau_\varepsilon \wedge t}\\
& =
x^{-m} - m \underline b^1 \int_0^{\tau_\varepsilon \wedge t}  X_s^{-m}  \, ds +  \int_0^{\tau_\varepsilon \wedge t} \big(\tfrac{m(m+1)\bar a^1}{2}-m\underline b^0\big) X_s^{-(m+1)}  \, ds + M_{\tau_\varepsilon \wedge t}.
\end{align*}
Taking expectations and using that $X\geq \varepsilon>0$ on $[\![0,\tau_\varepsilon]\!]$ yields, in view of \eqref{eq:m}, for all $t\geq 0$ that
\begin{align}\label{eq:Gronwall-prep1}
E[ (X_{ \tau_\varepsilon\wedge t})^{-m} ]
 & \le 
 x^{-m} 
- m \underline b^1 \int_0^{\tau_\varepsilon \wedge t} E[(X_{s})^{-m}] \,ds
& \le
 x^{-m} 
+ m |\underline b^1| \int_0^{\tau_\varepsilon \wedge t} E[(X_{s})^{-m}] \,ds
\end{align}
Moreover, 
 as $X\geq \varepsilon>0$ on $[\![0,\tau_\varepsilon]\!]$, we obtain that
\begin{equation*}
m |\underline b^1| \int_0^{\tau_\varepsilon \wedge t} E[(X_{s})^{-m}] \,ds 
=  m |\underline b^1| \int_0^{t} \ind{s \le \tau_\epsilon} E[(X_{\tau_\varepsilon\wedge s})^{-m}] \,ds
\leq
 m |\underline b^1|\int_0^t E[(X_{ \tau_\varepsilon \wedge s })^{-m}] \,ds.
\end{equation*}
This together with \eqref{eq:Gronwall-prep1} yields for any $t\geq 0$ that
\begin{align*}
E[ (X_{\tau_\varepsilon  \wedge t })^{-m} ] 
   \le 
   x^{-m} + m |\underline b^1|\int_0^t E[(X_{ \tau_\varepsilon\wedge s})^{-m} ] \,ds. 
\end{align*}
By means of Gronwall's inequality we obtain for all $t\geq 0$ that
\begin{align} \label{temp260}
    E[ (X_{t \wedge \tau_\varepsilon})^{-m} ] 
    & \le x^{-m}e^{m |\underline b^1| t}.
\end{align}
In particular, we obtain by Tschebyscheffs inequality for our fixed time $\cT>0$ that
\begin{align*}
   P(\tau_\varepsilon < \cT) &= P(X_{\tau_\varepsilon \wedge \cT} \le \varepsilon )
      = P((X_{\tau_\varepsilon \wedge \cT})^{-m} \ge \varepsilon^{-m} ) 
       \le \varepsilon^m E[(X_{\tau_\varepsilon \wedge \cT})^{-m}].
\end{align*}
 Inserting \eqref{temp260} and letting $\varepsilon$ tend to zero yields the claim.
The proof of Proposition~\ref{prop:positivity} is thus complete.
\end{proof}

\section{Dynamic programming}\label{sec:DPP-PDE}

One of the key insights for Markov processes is the deep link between Markov processes and their expectations to partial differential equations given by the \emph{Kolmogorov} equation. In this section we generalize this relation to the case with parameter uncertainty, i.e.~we develop the relation of the  non-linear affine process to a non-linear version of the Kolmogorov equation. The path we detail in this section uses dynamic programming and the results obtained in \cite{Nutz} and \cite{ElKarouiTan.13a}. The key to dynamic programming is a certain stability property under conditioning and pasting. As the non-linear affine processes we considered up to now always start from time $t=0$, we introduce the appropriate conditional formulations first.

For the remainder of the section, fix $\Theta$ as in \eqref{Theta} with associated $a^*$ and $b^*$ as in \eqref{def:AB}.
Denote by $\cA(t,x,\Theta)$ the semimartingale laws $P\in \fP_{\sem}^{\ac}$, such that
\begin{enumerate}[(i)]
    \item $P(X_t =x)=1$, and
    \item $P$ is affine-dominated on $(t,T]$ by $\Theta$.
\end{enumerate}

The following result yields measurability of the non-linear affine process starting at $t$ in $\omega(t)$.
\begin{lemma}\label{le:meas}
	The set $\big\{(\omega,t,P) \in \Omega\times [0,T]\times \fP(\Omega)\,\big|\, P\in \cA(t,\omega(t),\Theta) \big\}$ is Borel. 
\end{lemma}
\begin{proof}
	By Theorem 2.6 in \cite{NeufeldNutz2014}, the set $\fP_{\sem}^{\ac}$ is Borel, 
	which proves that 
	\begin{align*}
	& \ \big\{(\omega,t,P) \in \Omega\times [0,T]\times \fP_{\sem}^{\ac}\,\big|\, P(X_t =\omega(t)) =1 \big\}
	\end{align*}
	is Borel. 
	Moreover, Theorem 2.6 in \cite{NeufeldNutz2014}, also grants the existence of a Borel-measurable map
	$$ (P,\widetilde\omega,s) \mapsto (\beta_s^P(\widetilde\omega),\alpha_s(\widetilde \omega))$$
	such that $(\beta^P,\alpha)$ are the differential characteristics   of $X$ under $P$. Therefore, we obtain the Borel measurability of the set 
	\begin{align*}
	G&:=\bigg\{(\omega,t,P,\widetilde\omega,s) \in \Omega \times [0,T]\times \fP_{\sem}^{\ac}\times\Omega \times[0,T]\,
	   \bigg|\, 
	     s>t,\,   \\
	& \qquad \qquad P(X_t =\omega(t))=1,  (\beta^P_s(\widetilde\omega),\alpha_s(\widetilde\omega)) \in b^*(X_s(\widetilde\omega))\times a^*(X_s(\widetilde\omega))\bigg\}.
	\end{align*}
	Applying Fubini's theorem yields the Borel measurability of the set 
	\begin{equation*}
	G':=\Big\{(\omega,t,P,\widetilde\omega) \in \Omega\times [0,T]\times \fP_{\sem}^{\ac}\times\Omega\,\Big|\,\int_0^T \mathbf{1}_G(\omega,t,P,\widetilde\omega,s)\,\mathbf{1}_{(t,T]}(s)\,ds =T-t\Big\}.
	\end{equation*}
	Now, observe that
	\begin{align*}
	    \ \big\{(\omega,t,P) \in \Omega\times &[0,T]\times \fP(\Omega)\,\big|\, P\in \cA(t,\omega(t),\Theta)\big\}\\
	      &=  \ \big\{(\omega,t,P) \in \Omega\times [0,T]\times \fP_{\sem}^{\ac}\,\big|\,E^P[\mathbf{1}_{G'}(\omega,t,P,\cdot)]=1\big\}.
	\end{align*}
	and the right hand side is Borel measurable due to a monotone class argument as in \cite[Lemma~3.1]{NeufeldNutz2014}.
\end{proof}
\begin{lemma}\label{le:DPP-help}
	Fix  $(x,t) \in \R\times [0,T]$, $P \in \cA(t,x,\Theta)$ and a stopping time $\tau$ taking values in $[t,T]$. 
	\begin{enumerate}[(i)]
		\item  There exists a family of conditional probabilities $(P_{\omega})_{\omega \in \Omega}$ of $P$ with respect to $\ccF_\tau$ such that $P_\omega \in \cA(\tau(\omega),\omega_{\tau(\omega)},\Theta)$ for $P$-a.e. $\omega \in \Omega$. 
		\item Assume that there exists a family of probability measures $(Q_{\omega})_{\omega \in \Omega}$ such that $Q_\omega \in \cA(\tau(\omega),\omega_{\tau(\omega)},\Theta)$ for $P$-a.e. $\omega \in \Omega$, and the map $\omega \to Q_\omega$ is $\ccF_\tau$ measurable. Then, the probability measure $P\otimes Q$ defined by
		\begin{equation*}
		P\otimes Q(\,\cdot\,) =\int_{\Omega} Q_\omega(\,\cdot\,)\,P(d\omega)
		\end{equation*}
		is an element of $\cA(t,x,\Theta)$.
	\end{enumerate} 
\end{lemma}
\begin{proof}
	This follows directly from \cite[Theorem~2.1]{NeufeldNutz2017}, see also \cite[Lemma~4.6]{GuoTanTouzi.15}. 
\end{proof}

Now, fix a Borel measurable function $\psi:\ccO \to \R$ and  define the value function $v:[0,T]\times \ccO\to \R$ by
\begin{equation*}
v(t,x):=\sup_{P\in \cA(t,x,\Theta)} E^P[\psi (X_T)].
\end{equation*}
Using the above results, we obtain  the following dynamic programming principle.
\begin{proposition}\label{prop:DPP}
	Consider a proper family of non-linear affine processes with state space $\ccO$. For any $(t,x) \in [0,T]\times \ccO $ and any stopping time $\tau$ taking values in $[t,T]$, we obtain
	\begin{equation*}
	v(t,x)= \sup_{P\in \cA(t,x,\Theta)}E^P\big[v(\tau,X_\tau)\big].
	\end{equation*}
\end{proposition}
\begin{proof}
	The result follows from Theorem 2.1 in  \cite{ElKarouiTan.13b} (with $\overline{\mathcal{P}}_{t,x}=\cA(t,x,\Theta)$ in the notation of \cite{ElKarouiTan.13b})
  by noting that analyticity is implied by measurability as shown in Lemma \ref{le:meas} and the required stability assumptions have been shown in Lemma \ref{le:DPP-help}.
\end{proof}

\subsection{Continuity of the value function}
In the following, we show the continuity of the value function $v(t,x)$. To this end, introduce the constant
\begin{equation}\label{def:constant-K}
\mathcal{K}=|\underline{b}^0|+|\underline{b}^1|+|\bar{b}^0|+|\bar{b}^1|+\bar{a}^0+\bar{a}^1,
\end{equation}
which is finite by Assumption \eqref{Theta}. The following inequality is the cornerstone of the  results in this section.

\begin{lemma}\label{le:flow-estimate}
	Consider a proper family of non-linear affine processes with state space $\ccO$ and let $q\ge 1$. There exists an $0<\varepsilon\equiv\varepsilon(q)<1$ such that for all  $0<h\leq\varepsilon$, all $t \in[0,T-h]$ 
	and all $x\in \ccO$,
	we have
	\begin{equation*}
	\sup_{P \in \cA(t,x,\Theta)}E^P\Big[\sup_{0\leq s\leq h} |X_{t+s}-x|^q\Big]\leq C(h^q+h^{q/2})
	\end{equation*}
	for some constant $C=C(x,q)>0$ which may depend on $x$, but is independent of $h$ and $t$.
\end{lemma}
\begin{proof}
   Let $q \in [1,\infty)$.
   Consider $P \in \cA(t,x,\Theta)$ and denote by $X_s=x+B^P_s+M^P_s$, $s \ge t$, the semimartingale representation of $X$ with predictable finite-variation part $B^P$ and local martingale $M^P$. In the following we will repeatedly use the elementary inequality that 
   \begin{align}\label{eq:elementaryineq} 
        (a_1+a_2)^q \le 2^{q-1}( a_1^q + a_2^q) 
   \end{align}
   and denote $c_q:=2^{q-1}$.

   The  Burkholder--Davis--Gundy (BDG) inequality (see Theorem~4.1 in  \cite{RevuzYor1999}) together with Jensen's inequality and \eqref{eq:elementaryineq} yields for any $h \in [0,T-t]$ that
	\begin{align}\label{temp:308}
	   E^P\Big[\sup_{0\leq s \leq h}|X_{t+s}-x|^q\Big] &\leq c_q E^P\Big[\sup_{0\leq s \leq h}|M^P_{t+s}|^q\Big] + c_q E^P\Big[\sup_{0\leq s \leq h}|B^P_{t+s}|^q\Big]\\
	       &\leq c_q \widetilde C_q E^P\bigg[\Big(\int_t^{t+h} \alpha_u\,du\Big)^{\nicefrac q 2}\bigg] + c_qE^P\bigg[\Big(\int_t^{t+h} |\beta^P_u|\,du\Big)^q\bigg].\nonumber
	\end{align}
    Note that the constant $\widetilde C_q \ge 1$ from the BDG inequality does depend on $q$ only.

	Let $\widetilde{\mathcal{K}}:=1+\mathcal K$, where $\mathcal K$ is the constant defined in \eqref{def:constant-K}.
	Choose any $0<\varepsilon=\varepsilon(q)<1$
	small enough such that it satisfies 
	\begin{align} \label{cond:epsilon}
	1-c_q^3\widetilde C_q \widetilde{\mathcal{K}}^q  (\varepsilon^q+\varepsilon^{q/2})>0. \end{align} 
	Let us verify that such a fixed $\varepsilon$ satisfies the desired property: 
	by the very definition of $P \in \cA(t,x,\Theta)$, we have on $[t,t+h]$ that both $\alpha$ and $|\beta^P|$ are bounded from above by $\widetilde{\mathcal{K}}+\widetilde{\mathcal{K}}\sup_{0\leq s \leq h}|X_{t+s}|\ge 1$ since they are affine dominated. 
	This, together with  Jensen's inequality yields that
	\begin{align*}
	 E^P\bigg[\Big(\int_t^{t+h} \alpha_u\,du\Big)^{\nicefrac q 2}\bigg] 
	    &\leq  h^{q/2}E^P\bigg[\Big(\widetilde{\mathcal{K}} + \widetilde{\mathcal{K}} \sup_{0\leq s \leq h}|X_{t+s}|\Big)^{\nicefrac q 2}\bigg]\\
	    &\leq  h^{q/2}E^P\bigg[\Big(\widetilde{\mathcal{K}} + \widetilde{\mathcal{K}} \sup_{0\leq s \leq h}|X_{t+s}|\Big)^{q}\bigg]\\
	    &\leq  h^{q/2}c_q\bigg(\widetilde{\mathcal{K}}^q  + \widetilde{\mathcal{K}}^q E^P\Big[\Big(\sup_{0\leq s \leq h}|X_{t+s}|\Big)^q\Big]\bigg)\\
	    &\leq  h^{q/2}c_q^2\bigg(\widetilde{\mathcal{K}}^q  + \widetilde{\mathcal{K}}^q |x|^q + \widetilde{\mathcal{K}}^q E^P\Big[\Big(\sup_{0\leq s \leq h}|X_{t+s}-x|\Big)^q\Big]\bigg).
	\end{align*}
    In a similar way we obtain
	\begin{align*}
	 E^P\bigg[\Big(\int_t^{t+h} |\beta^P_u|\,du\Big)^q \bigg] 
	    &\leq  h^qE^P\Big[\Big(\widetilde{\mathcal{K}}+\widetilde{\mathcal{K}}\sup_{0\leq s \leq h}|X_{t+s}|\Big)^q\Big]\\
	    &\leq h^{q} c_q^2\bigg(\widetilde{\mathcal{K}}^q  + \widetilde{\mathcal{K}}^q |x|^q + \widetilde{\mathcal{K}}^q E^P\Big[\Big(\sup_{0\leq s \leq h}|X_{t+s}-x|\Big)^q\Big]\bigg).
	\end{align*}
	Inserting these inequalities into \eqref{temp:308}, considering $h\leq \varepsilon$ and noting that  $\tilde C_q \ge 1$ 
	implies that
	\begin{align}\label{eq:391}
	\lefteqn{E^P\Big[\sup_{0\leq s \leq h}|X_{t+s}-x|^q\Big]}\hspace{1.5cm}\\
	\leq & 	c_q^3 \, \widetilde C_q \widetilde{\mathcal{K}}^q(h^{q/2} + h^q) E^P\Big[\sup_{0\leq s \leq h}|X_{t+s}-x|^q\Big]
	       +c_q^3 \, \widetilde C_q  \widetilde{\mathcal{K}}^q (1+|x|^q)(h^{q/2} + h^q) \nonumber\\
	\leq & 	c_q^3 \, \widetilde C_q \widetilde{\mathcal{K}}^q(\varepsilon^{q/2} + \varepsilon^q) E^P\Big[\sup_{0\leq s \leq h}|X_{t+s}-x|^q\Big]
	       +c_q^3 \, \widetilde C_q  \widetilde{\mathcal{K}}^q (1+|x|^q)(h^{q/2} + h^q).\nonumber
	\end{align}
	Since $h\leq \varepsilon$ and we chose $0<\varepsilon<1$ such that \eqref{cond:epsilon} holds, we obtain for the constant 
	$C:=\frac{c_q^3 \, \widetilde C_q  \widetilde{\mathcal{K}}^q (1+|x|^q)}{1-c_q^3 \, \widetilde C_q \widetilde{\mathcal{K}}^q(\varepsilon^{q/2} + \varepsilon^q)}>0$, being  independent of $t,h,P$, that
	\begin{equation*}
	     E^P\Big[\sup_{0\leq s \leq h}|X_{t+s}-x|^q\Big] \leq C\,\big(h^{q/2}+h^q\big).
	\end{equation*}
	As $P \in \cA(t,x,\Theta)$ was  chosen  arbitrarily, the claim is proven.
\end{proof}
\begin{remark}\label{rem:martingale}
	The proof of Lemma~\ref{le:flow-estimate}, actually shows that for the corresponding $0<\varepsilon<1$, we have  for all  $0<h\leq\varepsilon$, all $t \in[0,T-h]$ and all $x\in \R$, that the local martingale part $(M^P_{t+s})_{0\leq s\leq h}$ restricted on $[t, t+h]$ is a true martingale for any $P \in \cA(t,x,\Theta)$.
\end{remark}
\begin{lemma}\label{le:value-funct-cont}
	Consider a proper family of non-linear affine processes with state space $\ccO$ and 
	let $\psi:\ccO\to \R$ be Lipschitz continuous with Lipschitz-constant $L_\psi$. Then the value function
	\begin{equation*}
	v(t,x):=\sup_{P \in \cA(t,x,\Theta)}E^P\big[\psi(X_T)\big]
	\end{equation*} 
	is jointly continuous. In particular, $v(t,\cdot)$ is  Lipschitz continuous with constant $L_\psi$ and $v(\cdot,x)$ is locally $1/2$-H\"older continuous. %
\end{lemma}
\begin{proof}
	For the Lipschitz-continuity of $v(t,\cdot)$, observe that for any $t$ %
	\begin{equation*}
	|v(t,x)-v(t,y)| \leq \sup_{P \in \cA(t,x,\Theta)}E^P\Big[|\psi(X_T)-\psi(y-x+X_T)|\Big]\leq L_\psi|y-x|.
	\end{equation*}
	For the locally $1/2$-H\"older continuity, let $t\in[0,T)$ and $0\leq u \leq T-t$ be small enough. Then the dynamic programming principle derived in Proposition~\ref{prop:DPP}, the Lipschitz continuity of $v(t,\cdot)$ and Lemma~\ref{le:flow-estimate} yield
	\begin{align*}
	|v(t,x)-v(t+u,x)|
	&=\Big|\sup_{P \in \cA(t,x,\Theta)}E^P\big[v(t+u,X_{t+u})-v(t+u,x)\big]\Big|\\
	&\leq L_\psi \sup_{P \in \cA(t,x,\Theta)}E^P\big[|X_{t+u}-x|\big]\\
	&\leq L_\psi C(x,1)\,\big(u^{1/2}+u\big)
	\end{align*}
	with constant $C(x,1)$ from Lemma \ref{le:flow-estimate}. Finally, consider a sequence $(t_n,x_n)$ converging to $(t,x)$. Then
	\begin{align*}
		| v(t_n,x_n) - v(t,x) | & \le |v(t_n,x_n)-v(t_n,x)| + | v(t_n,x)-v(t,x)| \\
		& \le L_\psi |x_n - x| + L_\psi C(x,1) \big( |t_n-t|^{\nicefrac 1 2} + |t_n-t| \big).
	\end{align*}
	Letting $n$ go to infinity yields the result.
\end{proof}

\section{The Kolmogorov equation}\label{sec:Kolmogorov}

In this section we provide the link between the non-linear affine process and the associated non-linear  Kolmogorov equation. More precisely, we relate the non-linear affine process to a (fully) non-linear partial differential equation (PDE). 
This is achieved by a probabilistic construction involving an optimal control problem on the canonical space of continuous paths where the controls are laws of affine-dominated semimartingales. To this end, note  that the affine process $X$ given in Equation \eqref{SDE} is uniquely characterized by its infinitesimal generator,
\begin{align}\label{eq:Kolmogorov}
   \ccL^\theta f(x)= (b^{0} + b^1x) \partial_x f(x)+ \half (a^0 + a^1 x^+)\partial_{xx}f(x). 
\end{align}
This is equivalent to $f(X_t)-f(x)-\int_0^t \ccL^\theta f(X_s)ds, \ 0 \le t \le T$ solving the martingale problem, see, e.g., Theorem 21.7 in \cite{Kallenberg2002}. 
Consider the state space $\ccO$ which will be either $\R$, $\R_{\ge 0}$ or $\R_{>0}$. Fix $\psi:\ccO \to \R$ and consider the fully non-linear PDE
\begin{equation} %
\begin{cases}-\partial_t v(t,x)-G\big(x,\partial_xv(t,x),\partial_{xx}v(t,x)\big) =0 \quad &\mbox{on }[0,T)\times \ccO, \label{eq:def:PDE}\\
\phantom{-\partial_t v(t,x)-G\big(x,\partial_xv(t,x),\partial_{xx}}\,v(T,x)=\psi(x) \quad &x \in \ccO,
\end{cases}
\end{equation}
where $G:\ccO\times \R \times \R \to \R$ is defined by
\begin{align}
G(x,p,q):=\sup_{(b^0,b^1,a^0,a^1)\in \Theta}\Big\{(b^0+b^1x)p+\frac{1}{2}(a^0+a^1x^+) q\Big\}. \label{eq:def:PDE-generator}
\end{align}
The function $-G$ satisfies the degenerate ellipticity condition and as $\Theta$ is compact, it is also continuous. Observe that the PDE defined in \eqref{eq:def:PDE} can be seen as non-linear affine PDE, since for $\theta:=(b^0,b^1,a^0,a^1)$  it is of the form
\begin{equation*}
\begin{cases}-\partial_t v(t,x)-\sup_{\theta \in \Theta}\ccL^\theta v(t,x)=0 \quad &\mbox{on }[0,T)\times \ccO, \nonumber\\
\phantom{-\partial_t v(t,x)-\sup_{\theta \in \Theta}\ccL}\,v(T,x)=\psi(x) & x \in \ccO.
\end{cases}
\end{equation*}

We write $C_b^{2,3}([0,T)\times \ccO)$ for the set of functions on $[0,T)\times \ccO$ having bounded continuous derivatives up to the second and third order in $t$ and $x$, respectively. An
upper semicontinuous function $u$ on $[0,T)\times \ccO$ will be called a viscosity subsolution of~\eqref{eq:def:PDE} if $u(T,\cdot)\leq \psi(\cdot)$ and
\begin{equation*}
-\partial_t \varphi(t,x)-G\big(x,D_x \varphi(t,x), D^2_{xx} \varphi(t,x)\big)\leq 0
\end{equation*}
whenever $\varphi \in C_b^{2,3}([0,T)\times \ccO)$ is such that $\varphi\geq u$ on $[0,T)\times \ccO$ and $\varphi(t,x)= u(t,x)$. %
The definition of a viscosity supersolution is obtained by reversing the inequalities and the semicontinuity. Finally, a 
continuous function is a viscosity solution if it is both sub- and supersolution.

We obtain a stochastic representation for the non-linear affine PDE \eqref{eq:def:PDE}.
\begin{theorem}\label{thm:PDE}
	Consider a proper family of non-linear affine processes with state space $\ccO$ and 
	let $\psi:\ccO \to \R$ be Lipschitz continuous. Then
	\begin{equation*}
	v(t,x):= \sup_{P\in \cA(t,x,\Theta)}E^P\big[\psi(X_T)\big], \quad x \in \ccO 
	\end{equation*}
	is a viscosity solution of the non-linear PDE in \eqref{eq:def:PDE}.
\end{theorem}
\begin{proof}
	The proof essentially follows the well-known standard arguments in stochastic control, see e.g. the proof of \cite[Proposition~5.4]{NeufeldNutz2017}.
	
	By Lemma \ref{le:value-funct-cont}, $v(t,x)$ is continuous on $[0,T)\times \R$, and we have $v(T,x)= \psi(x)$ by the definition of $v$. We show that $v$ is a viscosity subsolution of the non-linear affine PDE defined in \eqref{eq:def:PDE}; the supersolution property is proved similarly. We remark that in the subsequent lines within this proof, $C>0$ is a constant whose values may change from line to line.
	
	Let $(t,x) \in [0,T)\times \R$ and let $\varphi \in C^{2,3}_b([0,T)\times \R^d)$ be such that $\varphi\geq v$ and $\varphi(t,x)=v(t,x)$. By the dynamic programming principle obtained in Proposition~\ref{prop:DPP} we have for any $0<u<T-t$ that
	\begin{align}\label{eq:PDE-DPP}
	    0&=\sup_{P\in \cA(t,x,\Theta)}E^P\big[v(t+u,X_{t+u})-v(t,x)\big] \nonumber\\
	     &\leq \sup_{P\in \cA(t,x,\Theta)}E^P\big[\varphi(t+u,X_{t+u})-\varphi(t,x)\big]. 
	\end{align}
	Fix any $P\in \cA(t,x,\Theta)$, denote as above by  $(\beta^P,\alpha)$ the differential characteristics of the continuous semimartingale $X$ under $P$,
  and denote by $M^P$ the $P$-local martingale part of the $P$-semimartingale $X$.
	Then, It\^o's formula yields
	\begin{align}
	\varphi (t+u,&X_{t+u}) - \varphi(t,x) =
	\int_0^u \partial_t \varphi(t+s,X_{t+s}) \,ds
	+
	\int _0^u \partial_x \varphi(t+s,X_{t+s}) \,dM^P_{t+s} \nonumber\\
	& + \int_0^u \partial_x \varphi(t+s,X_{t+s}) \beta_{t+s}^P \,ds 
	+ \frac{1}{2} \int_0^u \partial_{xx} \varphi (t+s,X_{t+s}) \alpha_{t+s} \,ds. \label{eq:PDE1}
	\end{align}
	As $\varphi \in C_b^{2,3}([0,T)\times \R)$,  $\partial_x \varphi$ is uniformly bounded, thus by Remark~\ref{rem:martingale}, we see that for small enough $0<u<T-t$ the local martingale part in \eqref{eq:PDE1} is in fact a true martingale, starting at $0$. In particular, its expectation vanishes. %
	The next step is to estimate the expectation of the other terms. In this regard, note that
	\begin{align}
	 E^P\bigg[ \int_0^u &\partial_x \varphi(t+s,X_{t+s}) \beta^P_{t+s}\, ds\bigg]  \nonumber\\
	\leq & \ \int_0^u E^P \bigg[ \big| \partial_x \varphi(t+s,X_{t+s}) - \partial_x \varphi (t,x)\big| \, |\beta^P_{t+s}| +
	\partial_x \varphi(t,x) \beta^P_{t+s} \bigg]\, ds. \label{eq:PDE2}
	\end{align}
	Since $\varphi \in C^{2,3}_b$, $\partial_x \varphi$ is Lipschitz. Hence we obtain with the constant $\cK$ from Equation \eqref{def:constant-K} together with 
	 Lemma~\ref{le:flow-estimate} that for small enough $u$,
	\begin{align}
	 \ \int_0^u E^P \Big[ \big| &\partial_x \varphi(t+s,X_{t+s}) - \partial_x \varphi (t,x)\big| \cdot |\beta_{t+s}^P| \Big]\, ds \nonumber\\
	\leq & \  C\int_0^u E^P \Big[ \big(s+\sup_{0\leq v\leq u}|X_{t+v}-x|\big)  \cdot |\beta^P_{t+s}| \Big]\, ds \nonumber\\
	\leq & \  C\int_0^u E^P \Big[ \big(s+\sup_{0\leq v\leq u}|X_{t+v}-x|\big)\, \big(\mathcal{K} +\mathcal{K} \sup_{0\leq v\leq u}|X_{t+v}|\big)\Big]\, ds \nonumber\\
	\leq & \  C\int_0^u E^P \Big[ \big(s+\sup_{0\leq v\leq u}|X_{t+v}-x|\big)\, \big(\mathcal{K} +\mathcal{K} |x| +\mathcal{K} \sup_{0\leq v\leq u}|X_{t+v}-x|\big)\Big]\, ds \nonumber\\
	\leq & \ C\big(u^3 +u^{5/2}+ u^2 + u^{3/2}\big). \label{eq:PDE3}
	\end{align}
	Inserting \eqref{eq:PDE3} into \eqref{eq:PDE2} yields
	\begin{align} \label{eq:PDE4}
	  	   \lefteqn{ E^P \bigg[ \int_0^u \partial_x \varphi(t+s,X_{t+s}) \beta^P_{t+s}\, ds\bigg]} \hspace{2cm} \nonumber\\
	   & \leq \int_0^u E^P\Big[\partial_x \varphi(t,x) \,\beta^P_{t+s}\Big]\, ds + C\big(u^3 +u^{5/2}+ u^2 + u^{3/2}\big). 
	\end{align}
	The same argument applied to $\partial_{xx} \varphi$ leads to
	\begin{align} \label{eq:PDE5}
	\lefteqn{E^P\bigg[ \int_0^u \partial_{xx} \varphi(t+s,X_{t+s})\, \alpha_{t+s}\, ds\bigg] } \hspace{2cm} \nonumber\\
	    & \leq \int_0^u E^P\Big[\partial_{xx} \varphi(t,x)\, \alpha_{t+s}\Big]\, ds + C\big(u^3 +u^{5/2}+ u^2 + u^{3/2}\big).
	\end{align}
	Moreover, by a similar calculation, we have
	\begin{align}
	 E^P\bigg[ \int_0^u \partial_t &\varphi(t+s,X_{t+s})\, ds\bigg]  \nonumber\\
	\leq & \ \int_0^u \partial_t \varphi (t,x) \,ds + \int_0^u E^P \Big[ \big| \partial_t \varphi(t+s,X_{t+s}) - \partial_t \varphi (t,x)\big| \Big] \, ds \nonumber\\
	\leq & \ \int_0^u \partial_t \varphi (t,x) \,ds + C\int_0^u E^P \Big[s+\sup_{0\leq v\leq u}|X_{t+v}-x| \Big]\, ds \nonumber\\
	\leq & \ \int_0^u \partial_t \varphi (t,x) \,ds + C\big(u^2 + u^{3/2}\big). \label{eq:PDE6}
	\end{align}
	As above we write $\theta:=(b^0,b^1,a^0,a^1)$ for an element in $\Theta$.
	Then, by taking expectations in \eqref{eq:PDE1} and using \eqref{eq:PDE2}--\eqref{eq:PDE6}  yields
	\begin{align}
	\lefteqn{ E^P\Big[\varphi (t+u,X_{t+u})- \varphi(t,x)\Big] 
	      \leq  C\big(u^3 +u^{5/2}+ u^2 + u^{3/2}\big) } \quad \nonumber\\
	      &+ \int_0^u \Big(\partial_t \varphi (t,x) + E^P\big[
	       \partial_x \varphi(t,x) \,\beta^P_{t+s}
	       + \partial_{xx} \varphi(t,x)\, \alpha_{t+s}\big] \Big) \,ds \ \nonumber\\ 
          &\le  C\big(u^3 +u^{5/2}+ u^2 + u^{3/2}\big) + u \partial_t \varphi (t,x)\nonumber \\
          &+ \int_0^u E^P\Big[ \sup_{\theta \in \Theta} \Big\{(b^0+b^1 X_{t+s})\,\partial_x \varphi(t,x) + \frac{1}{2} (a^0+a^1 X_{t+s}^+)\,\partial_{xx} \varphi(t,x) \Big\}\Big] .\label{eq:PDE7}
	\end{align}
	Here the supremum turns out to be  $G(X_{t+s},\partial_x \varphi(t,x),\partial_{xx}\varphi(t,x))$. 
	Note that by the very definition of $G$,
	\begin{align*}
		G(X_{t+s},p,q) 
		& \le G(x,p,q) + \sup_{\theta \in \Theta}\Big\{ |b^1| \, |X_{t+s}-x| \, | p| +
		  |a^1| \, |X_{t+s}-x| \, | q|
		\Big\}.
	\end{align*}
	Therefore, by using that $\varphi \in C^{2,3}_b$, the definition of the constant $\cK$ in \eqref{def:constant-K} and Lemma~\ref{le:flow-estimate}, we have
	\begin{align}
    	\lefteqn{\int_0^u E^P\Big[ G(X_{t+s},\partial_x \varphi(t,x),\partial_{xx}\varphi(t,x)) \Big\}\Big] \,ds }\hspace{2cm} \nonumber \\
	    \leq & \ u  G(x,\partial_x \varphi(t,x),\partial_{xx}\varphi(t,x))
	    + u C\cK E^P\big[|X_{t+s}-x|\big] \nonumber\\
	    \leq & \ u  G(x,\partial_x \varphi(t,x),\partial_{xx}\varphi(t,x)) +  C\cK \big(u^2 +u^{3/2}\big).  \label{eq:PDE8}  
	\end{align}
	Combining \eqref{eq:PDE7}--\eqref{eq:PDE8} yields
	\begin{align}
 E^P\Big[\varphi (t+u,X_{t+u})- \varphi(t,x)\Big] 
	    &\leq  \  u \partial_t \varphi (t,x) +  u  G(x,\partial_x \varphi(t,x),\partial_{xx}\varphi(t,x))
	    \nonumber \\
	    & \ +  C\big(u^3 +u^{5/2}+ u^2 + u^{3/2}\big). \label{eq:PDE9}  
	\end{align}
	for some constant $C>0$ which is independent of $P$. As the choice of $P \in \cA(t,x,\Theta)$ was arbitrary, we deduce from
	\eqref{eq:PDE-DPP} that
	\begin{align}
	0 &\leq \sup_{P\in \cA(t,x,\Theta)}E^P\big[\varphi(t+u,X_{t+u})-\varphi(t,x)\big]\nonumber\\
	&\leq u \partial_t \varphi (t,x) +  u  G(x,\partial_x \varphi(t,x),\partial_{xx}\varphi(t,x)) 
	 +  C\big(u^3 +u^{5/2}+ u^2 + u^{3/2}\big). \label{eq:PDE10}
	\end{align}
	By dividing first in \eqref{eq:PDE10} by $-u$ and then let $u$ go to zero, we obtain that
	\begin{equation*} 
	- \partial_t \varphi (t,x) - G(x,\partial_x \varphi(t,x),\partial_{xx}\varphi(t,x)) \leq 0,
	\end{equation*}
	which proves that $v$ is indeed a viscosity subsolution as desired. 
\end{proof}

\subsection{Uniqueness}
Uniqueness in our framework is not covered by standard arguments as in \cite{fleming-soner-06} or \cite{crandall-ishii-lions-92}, since the diffusion coefficient does not satisfy a global Lipschitz condition. 
This is a well-known difficulty already discussed in \cite{Feller}. To overcome this, we have to distinguish the two cases where the state space is either $\R$ or $\R_{>0}$.

We begin with the general case which covers the non-linear Vasi\v cek-CIR model discussed in detail in Section \ref{sec:non-linearVasicekCIR}. In this case we do not decide a priori whether a Vasi\v cek or a Cox-Ingersoll-Ross (CIR) model takes place and therefore consider the full space $\R$ as state space. To achieve this, we assume throughout a non-vanishing volatility, i.e.~$\underline a^0 >0$. When the process reaches zero from above, this avoids that a deterministic behaviour on $\R_{<0}$ takes place and that singularities arise. 
On the other side, we do not need any assumptions on $a^1$.

\begin{proposition}\label{thm:PDE-uniquenessI}
	Assume that $\underline a^0>0$ and let $\ccO=\R$. Then, 
	for any given continuous and bounded function
	 $\psi:\R \to \R$, 
    the non-linear  PDE introduced in \eqref{eq:def:PDE}--\eqref{eq:def:PDE-generator}
	admits at most one viscosity solution $v(t,x)$ on $[0,T]\times\R$ satisfying the terminal conditions
	\begin{align*}
	v(T,x)&=\psi(x), \quad x\in \R.
	\end{align*} 
\end{proposition}

\begin{proof}
   Uniqueness in this case follows by the observation that the coefficients are Lipschitz once $\underline a^0>0$. 
   Indeed, then following the standard procedure detailed in Section V.9 in \cite{fleming-soner-06} allows to extend the uniqueness results from Corollary V.8.1 therein to an unbounded domain as considered here.
\end{proof}

On the other side if we consider the case where $\R_{>0}$ is the state space, we will necessarily require $\underline a^0 = \bar a^0 = 0.$ Then, the (bounds on the) diffusion coefficient do no longer satisfy a global Lipschitz property and the standard methodology can not be applied. To the best of our knowledge, only \cite{ConstantiniPapiDIppoliti} and \cite{Amadori2007} treat this setting while we will apply here the techniques from the second article. 

In this regard, let
\begin{align}
h_\epsilon(x) &= \begin{cases}
\max\Big( \frac 1 {\log(x)}, 1 \Big) & \epsilon=0 \\
\max\Big( \frac 1 {x^\epsilon},1 \Big) & \text{otherwise.}
\end{cases}
\end{align}

\begin{proposition}\label{thm:PDE-uniquenessII}
	Assume that $\underline a^0=\bar a^0=0$, $\underline b^0 \ge \nicefrac{\bar a^1} 2>0$, let $\ccO=\R_{>0}$, and
	consider a Lipschitz-continuous  $\psi:\ccO \to \R$. 
	\begin{enumerate}[(i)]
	\item Then,  \eqref{eq:def:PDE}--\eqref{eq:def:PDE-generator}
	has a one and only one viscosity solution $v$ on $[0,T]\times\ccO$ satisfying
	\begin{align}\label{eq:694}
	\sup_{(t,x) \in [0,T]\times \ccO} \frac{|v(t,x)|}{1+x} < \infty.
	\end{align}
	\item Let $\epsilon = \nicefrac{2 \underline b^0}{\bar a^1} -1 >0$ and suppose that there is $\rho,\rho'>0$
	and $\epsilon'\in [0,\epsilon)$ such that $\frac{|\psi(x)|}{1+	|x|}$ is bounded on $(0,\infty)$ and
	\begin{align}
	  x \mapsto \frac{|\psi(x)|}{h_{\epsilon'}(x)}
	\end{align}
	is bounded from above respectively in the set $(\rho,\infty)$ and $(0,\rho')$. Then  \eqref{eq:def:PDE}--\eqref{eq:def:PDE-generator}
	has one and only one viscosity solution $v(t,x)$ on $[0,T]\times\ccO$ satisfying
	\begin{align}
	   & \limsup_{x \to \infty} \sup_{t \in [0,T]} \frac{|v(t,x)|}{x^2} = 0, \quad
        \lim_{x \to 0} \sup_{t \in [0,T]} \frac{|v(t,x)|}{h_\epsilon(x)} = 0.
	\end{align}
		\end{enumerate}
\end{proposition}

\begin{proof}
Positivity follows readily from Proposition \ref{prop:positivity}. This will allow to treat the non-linear PDE  in \eqref{eq:def:PDE}--\eqref{eq:def:PDE-generator} on the state space $\ccO=\R_{>0}$. The key for uniqueness is that on compact subsets of $\R_{>0}$  the square-root \emph{is} Lipschitz continuous. 

To begin with, note that existence of a solution under the Lipschitz property of $\psi$ follows from Theorem \ref{thm:PDE}.
Furthermore, Theorem 4 in \cite{Amadori2007} yields the desired uniqueness in our case. We recall this result together with its assumptions in Appendix \ref{app:COMP}. The claim the follows from Theorem \ref{thm:Amadori}.
\end{proof}

\begin{remark}
Already the results in the case without uncertainty in \cite{ConstantiniPapiDIppoliti} show that these results do not generalize to $\R_{\ge 0}$, because then the Lipschitz property on compact subsets which is crucially used in the proof will no longer be satisfied. Moreover, Example 1 in \cite{Amadori2007} shows that the condition $\underline b^0 \ge \nicefrac{\bar a^1} 2>0$ is indeed necessary for uniqueness.
\end{remark}

It is time to accommodate the established results with some first  motivating examples. 
We will begin with classical affine processes under parameter uncertainty, leading to non-linear affine processes. Thereafter we show how to extend beyond this case and introduce classes of non-linear affine processes which do not have a classical counterpart. 

\subsection{The non-linear Vasi\v cek model}\label{non-linearVasicek}
The first example of an affine model is the so-called Vasi\v cek model, see for example \cite{Filipovic2009}. It is a Gaussian Ornstein-Uhlenbeck process and is obtained as the strong solution of the SDE in Equation \eqref{SDE} by considering $a^1=0$. Introducing parameter uncertainty we arrive at the non-linear affine process with 
$[\underline a^1,\bar a^1]=\{0\}$. We call this case the \emph{non-linear Vasi\v cek model}.

While in the case with no parameter uncertainty, this model can be characterized efficiently by its Fourier transforms and the associated Riccati equation, this is no longer  possible here (except for the special case where $\underline b^1 = \bar b^1$, see Remark \ref{remarktoprop:icx} below) and one has to rely on numerical techniques.  In one dimension, this does not at all pose a problem,  see for example \cite{Heider}. To illustrate this, we solve the equation for the simplest pay-off, $f(x)=e^x$. The result is shown in Figure \ref{fig:Vasicekexp}. 

\begin{figure}[h]
\begin{tabular}{cc}
	\begin{minipage}{3cm}
	\begin{tabular}{cc}\toprule & value \\ \midrule
	$\overline b_0$ & 0.15\\
	$\overline b_1$ & -0.5\\
	$\underline b_1$ & -3\\
	$\overline a_0$ & $0.02$\\
	\bottomrule
	\end{tabular} \end{minipage} & \hspace{-15mm}
	\begin{minipage}{9.5cm}
	 \begin{overpic}[width=11cm, trim=0 0 0 10, clip]{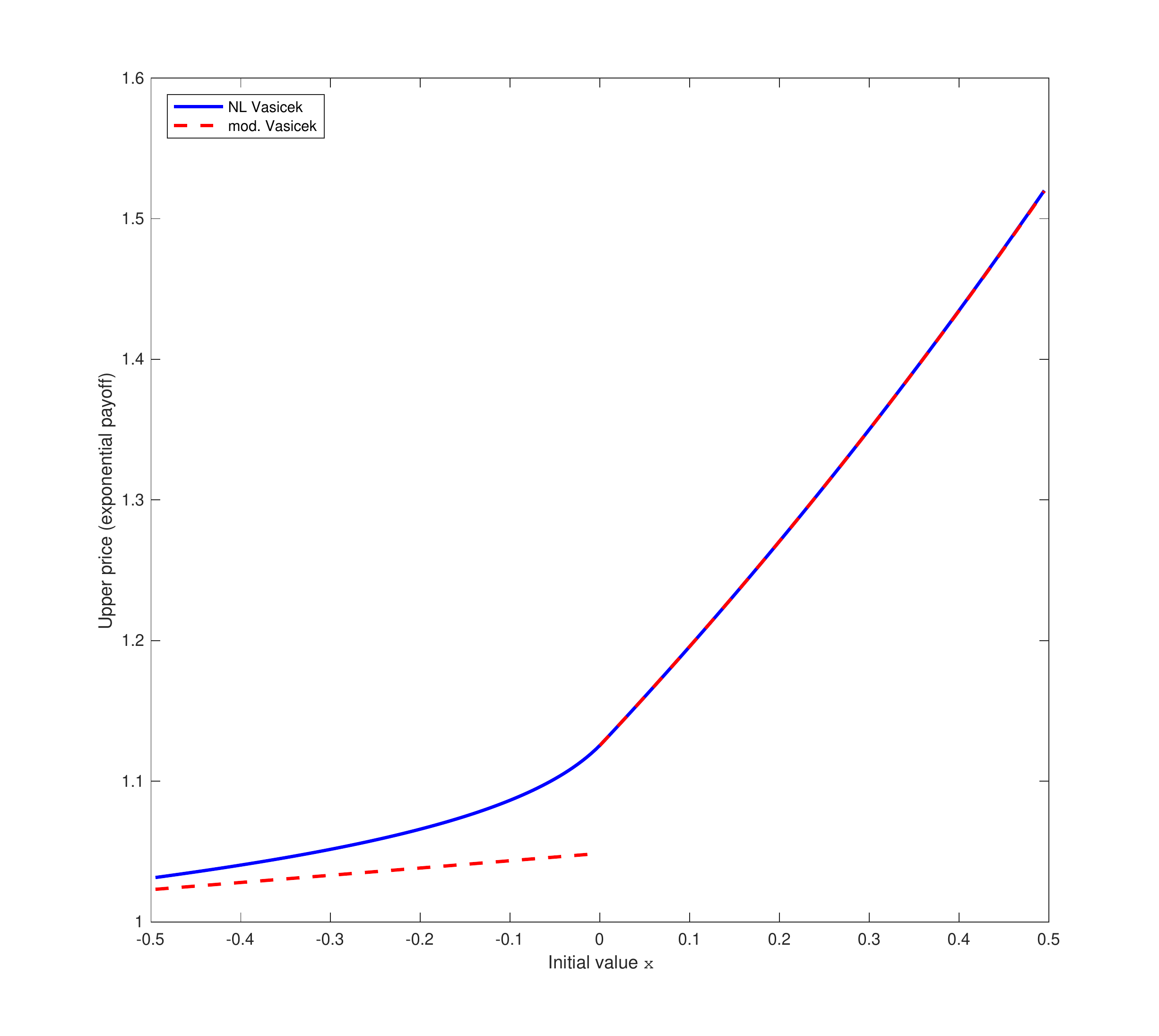} \end{overpic}  \end{minipage}
\end{tabular}
\caption{\label{fig:Vasicekexp} This figure shows the solution of the non-linear Kolmogorov equation for the non-linear Vasi\v cek-model with boundary condition $f(x)=e^x$. The dashed line shows the solutions for the Vasi\v cek model with parameter-set $(\overline b_0, \underline b_1, \overline a_0) $ (on $\R_{<0}$) and $(\overline b_0, \overline b_1, \overline a_0) $ (on $\R_{\ge 0}$) and illustrate the non-linearity in the solution due to the parameter uncertainty. Note that on the positive halfline the curves overlap in the figure.}
\end{figure}

This example combines and encompasses the following two well-known non-linear processes: $g$-Brownian motion and $G$-Brownian motion, see \cite{Peng.97} and \cite{D-Peng}, for example, and \cite{NeufeldNutz2017} for the case with jumps. In the following we will also elaborate on the case where explicit solutions can be obtained, see Proposition \ref{prop:icx} and Remark \ref{remarktoprop:icx}.

\subsection{The non-linear CIR model}\label{sec:CIR}
While the Gaussianity of the Vasi\v cek model immediately implies that the process becomes negative with positive probability, this is inappropriate for various applications, e.g.~in credit risk. There,  the considered affine process models an intensity, which by definition has to be non-negative. Also positive interest rates were a must of a model before the recent crises in 2008-2010. The Cox-Ingersoll-Ross (CIR) model serves as an affine model with state space $\R_{\ge 0}$ and therefore satisfies these needs. It is obtained by choosing $a^0=0$ and $a^1>0$ in the SDE in  \eqref{SDE}. 

The CIR model under parameter uncertainty, which we call the non-linear CIR model, is obtained by considering the state space $\ccO=\R_{>0}$ and assuming that $\underline a^0=\bar a^0=0$, $\underline b^0 \ge \nicefrac{\bar a^1} 2>0$. It is remarkable how far-reaching the positivity of $X$ will be: indeed, a first look at the non-linear Kolmogorov equation already reveals that for increasing and convex functions the supremum in Equation \eqref{eq:def:PDE-generator} will be attained by the upper bounds of the coefficients. 

The Fourier-transform which is  lacking  monotonicity (as $e^{iux}$ does) turns out to loose tractability in the non-linear case considered here. The Laplace transform does not suffer from this problem and we show in the following that in important special cases it can be computed explicitly. However, inversion techniques can no longer be applied in the non-linear setting and the Laplace transform merely serves as a prime example of an increasing convex function for which the non-linear expectations can be computed explicitly in special cases (and for decreasing concave functions of course).

\begin{remark}
In principle, the non-linear CIR model can be extended to the whole space, i.e.~ choosing $\ccO=\R$ is possible, since in Equation \eqref{def:AB}, negative values pose no problem. In the negative halfline the dynamics of the model have no diffusive part and only moves through the drift (which of course can still be stochastic). This could happen with a positive (upper) mean-reversion level and by starting from a negative value.
\end{remark}

We begin by a general result for affine processes which classifies when the classical representation via Riccati equations still holds. 

Recall from  \eqref{eq:affineb*}, that $b^*(x)$ and $a^*(x)$ are actually intervals.  
Define the upper bound of  $a^*(x)$  by
\begin{align}\label{astar}
    \bar a(x) & :=   \bar a^0 + \bar a^1 x^+. 
\end{align}
Note that the upper bound of $b^*(x)$ is given by 
\begin{align}\label{bstar} 
    \bar b(x) & :=  \bar b^{0} + \underline b^1 x \ind{x <0} + \bar b^1 x \ind{x \ge 0}. 
\end{align}
Letting $ \bar B^{1,x} =  \underline b^1  \ind{x <0} + \bar b^1  \ind{x \ge 0}, $ we obtain that 
$\bar b(x) = \bar b^{0} + \bar B^{1,x} x.$ Note that for $x \neq y$ we may no longer have that $\bar b(x) = \bar b^{0} + \bar B^{1,y} x$.

\begin{proposition}\label{prop:icx}
Consider a non-linear affine process $\cA(x,\Theta)$ and assume that for all $P \in \cA(x,\Theta)$
	\begin{align}\label{eq:barAB}
     \beta^P_t \le \bar b^{0} + \bar B^{1,x} X_t,
	\end{align}
	$dP\otimes dt$-almost everywhere for $0 \le t \le T$. Moreover, assume either that $\underline a_1 = \bar a_1 = 0$ or that for all $P\in\cA(x,\Theta)$, $X_t \ge 0$ $P \otimes dt$-a.e. 
  If there exists a $\bar P\in\cA(x,\Theta)$ and a $\bar P$-$\bbF$-Brownian motion $W$ such that the canonical process under $\bar P$ is the unique strong solution of 
	\begin{align} \label{SDEAB}
	dX_t = (\bar b^{0} + \bar B^{1,x} X_t)dt  + \sqrt{  \bar a(X_t)} dW_t, \qquad  X_0=x,\end{align}
	then, for all $u\ge 0$ and $0 \le t \le T$,
	$$ \sup_{P \in \cA(x,\Theta)}E^P\big[e^{uX_t}\big] = \exp\big( \phi(t,u) + \psi(t,u) x \big), $$
	where $\phi,\ \psi$ solve the Riccati equations
 \begin{align} \label{Ric1}
    \partial_t \phi(t,u) &=  \frac{1}{2} \bar a^{0} \psi(t,u)^2 + \bar b^{0} \psi(t,u) , \qquad \phi(0,u)= 0\\
    \partial_t \psi(t,u) &=  \frac{1}{2} \bar a^{1} \psi(t,u)^2 + \bar B^{1,x} \psi(t,u), \qquad \psi(0,u)= u.
    \label{Ric2}
 \end{align}
\end{proposition}
This result is obtained by showing that the supremum in the non-linear expectation is obtained by the maximal semimartingale law $\bar P$ which corresponds to an affine process with parameters $\bar a^{0}, \bar a^1, \bar b^0, \bar B^{1,x}$. Inspection of the proof shows that this property also holds when $e^{ux}$ is replaced by any other increasing and convex function (a Call payoff, for example). An analogous formulation for $u<0$ of course also holds. 
Furthermore, it is interesting to see that the Riccati equations in Equations \eqref{Ric1} and \eqref{Ric2} can be replaced by respective versions thereof. 

Note that the assumption $X_t \ge 0$ $P \otimes dt$-a.e.~is implied by assuming  $x>0$ together with $\underline a^0=\bar a^0=0$ and  $\underline b^0 \ge \nicefrac{\bar a^1} 2>0$ due to Proposition \ref{prop:positivity}.

\begin{proof}
The claim follows  by an application of Theorem 2.2 in \cite{BergenthumRueschendorf2007}. This needs validity of the so-called propagation of order (PO) property and we give a detailed account of this in Appendix \ref{app:PO}. Proposition \ref{prop:PO} in particular yields that the PO property is satisfied for increasing and convex (decreasing and concave) functions when compared to an affine process. 

Let $P\in\cA(x,\Theta)$. By assumption there exists a $\bar P \in \cA(x,\Theta)$ and a $\bar P$-$\bbF$-Brownian motion $W$, such that $X$ is the unique strong solution of the  stochastic differential equation  \eqref{SDEAB} under $\bar P$.
Since $e^{ux}$ for $u \ge 0$ is increasing and convex and \eqref{eq:barAB} holds, Theorem 2.2 in \cite{BergenthumRueschendorf2007} yields that
$$ E^P\big[e^{uX_t}] \le  E^{\bar P}\big[ e^{uX_t}\big]. $$
Since $P\in\cA(x,\Theta)$ was arbitrary and, under $\bar P$, $X$ is an affine process with parameters $ \bar a^{0}, \bar a^1, \bar b^0, \bar B^{1,x}$,  the affine representation follows and the Riccati equations in \eqref{Ric1} can be obtained directly from  Theorem 10.1 in \cite{Filipovic2009}.
\end{proof}

\begin{remark}\label{remarktoprop:icx}
It can be easily checked that the conditions for Proposition \ref{prop:icx} hold in the following two cases:
\begin{enumerate}
\item The non-linear Vasi\v cek model with state space $\ccO=\R$ and $\bar b^1 = \underline b^1$,
\item the non-linear CIR model on the state space $\ccO=\R_{>0}$.
\end{enumerate}
\end{remark}

By the above result we can use the classical Fourier-inversion technique for these affine processes when pricing increasing and convex payoffs (like Call options) or decreasing and concave ones, see Section 10.3 in \cite{Filipovic2009} for examples and details in this direction. In Example \ref{ex:Heston} we will sketch an application in a Heston model with parameter uncertainty in the stochastic volatility.

\section{An It\^o-formula for non-linear affine processes}\label{sec:Ito}

In this section we will construct new processes from non-linear affine processes by simple transformations. The main tool for this will be a suitable formulation of the  It\^o-formula in our setting. 

Consider a twice continuously differentiable function $F \in C^2(\R)$. If we start from a non-linear affine process $\cA(x,\Theta)$ and consider $\tilde X:=F(X)$ then for any $P\in \cA(x,\Theta)$ the process $\tilde X$ is a $P$-semimartingale and we denote its (differential) semimartingale characteristics by $\tilde \alpha$ and $\tilde \beta^P$ (starting from $\alpha$ and $\beta^P$ from Equality \eqref{eq:semimartchar}). In this section we answer the question if the non-linear process $\tilde X$ itself, i.e.~the associated semimartingale laws can be studied independently of $X$. This corresponds one-to-one to the question if there exist an independent formulation of the non-linear process $\tilde X$. The following proposition gives a positive answer to this question.

We define the  interval-valued functions $a^F$ and $b^F$ by 
\begin{align}
a^F(x)& := [ (F^\prime(x))^2 (\underline a^0 + \underline a^1 x^+), (F^\prime(x))^2 (\bar a^0+\bar a^1 x^+) ]
\end{align}
and
\begin{align}\label{eq:bF}
	b^F(x) &:= 
	\Big[ \inf_{(\beta,\alpha) \in b^*(x)\times a^*(x)} \Big(F^\prime(x) \beta + \half F^{\prime\prime}(x) \alpha\Big),
	      \sup_{(\beta,\alpha) \in b^*(x)\times a^*(x)} \Big(F^\prime(x) \beta + \half F^{\prime\prime}(x) \alpha\Big) \Big].
\end{align}
The non-linear process $\tilde X$ inherits certain bounds from $X$ which is characterized in the following proposition.

\begin{proposition}
	\label{Ito}
	Let $\mathcal{A}(x,\Theta)$ be a non-linear affine process and $F\in C^2$. Then, for every $P\in \cA(x,\Theta)$, $\tilde X=F(X)$ is a $P$-semimartingale with differential characteristics $\tilde \alpha$ and $\tilde \beta^P$ satisfying
	\begin{align}
	   \tilde \alpha_s  & \in a^F(X_s), \\
	   \tilde \beta^P_s & \in b^F(X_s).
	\end{align}
\end{proposition}

\begin{proof}
	Let $t\in [0, T]$. By definition, $P \in \mathcal{A}(x,\Theta)$	 implies that  
	\begin{equation*}
		X_t= X_0 + \int_{0}^{t}\beta_s^P ds +  M_t^P
	\end{equation*}
	where $\beta_s^P \in b^*(X_s)$, $\alpha_s \in a^*(X_s)$ and $M^P$ is the continuous local martingale part in the $P$-semimartingale decomposition of $X$. As previously, we denote $\langle M^P \rangle=\int_0^\cdot \alpha_s ds$.
	Since $F\in C^2(\R)$, the It\^o formula yields that
	\begin{equation*}
		F(X_t)= F(X_0) + \int_0^t \big(F^\prime(X_s) \beta_s^P + \frac{1}{2} F^{\prime \prime}(X_s)\alpha_s\big) ds + \int_0^t  F^\prime(X_s) dM^P_s
	\end{equation*}	
	and, hence, 
	\begin{align}
		\tilde{\beta}^{{P}}_s &= F^\prime(X_s) \beta^P_s + \frac{1}{2}F^{\prime \prime}(X_s)\alpha_s  \label{temp933}\\
		\tilde{\alpha}_s &=  (F^\prime(X_s))^2 \alpha_s. \notag
	\end{align}
	Now the properties $\tilde \alpha_s   \in a^F(X_s)$ and $\tilde \beta^P_s  \in b^F(X_s)$ can be checked directly. 
\end{proof}

\begin{remark}
Intuitively, the above result allows to construct the non-linear process $\tilde X=F(X)$ when $X$ is non-linear affine. The new bounds for the (differential) semimartingale characteristics are given by $b^F(X)$ and $a^F(X)$, respectively. However, the drift and volatility of $\tilde X$ now relate to each other, which often gives a substantially smaller class in comparison to all semimartingale laws whose drift and volatility stay in $b^F(X_t)$ and $a^F(X_t)$. 
\end{remark}

In general, (non-linear) affine processes are stable under affine transformation. 
The following example shows that, we may even consider the non-linear transformation $F(x)=x^2$, at least in some special cases.

\begin{example}
	Let $\cA(x,\Theta)$ be a non-linear Vasi\v cek model satisfying $\bar b^0=\underline b^0 =0$, and $\tilde X = F(X)=X^2$. 
	We apply Proposition \ref{Ito}: first, note that since $F^{\prime \prime}=2>0$,
	\begin{equation*}
		{b}^F(x) = 
			 [2x^2\underline{b}^1+ \underline{a}^0, 2x^2\bar{b}^1+ \bar{a}^0], 
	\end{equation*}	
	and  ${a}^F(x)=  [4 x^2 \underline{a}^0, 4x^2 \bar{a}^0 ]$.   
  Then, $b^F$ and $a^F$ can even be written as functions of $\tilde X=X^2$. This would not be the case if $\bar b^0=\underline b^0 =0$ does not hold, since $b^F$ would depend on $x$ (which is not a function of $x^2$). 
  Under this observation we may directly study the semimartingale characteristics of $\tilde X$. Replacing $x^2$ by $\tilde x$ in $b^F$ and $a^F$ we indeed observe an affine structure and it is tempting to conjecture that we obtained a non-linear CIR model. In general, this is not the case: for simplicity, choose $\underline a^0 = \underline b^1 = 0$ and $\bar a^0=\bar b^1 = 1$ and $x=1$. Then $b^F(1)=[0,3]$ and $a^F(1)=[0,4]$. For a non-linear CIR model, any choices of $(\tilde \beta, \tilde \alpha)$ in $b^F \times a^F$ should be possible. Now choose, say, $\tilde \alpha=4$ (corresponding to a maximar volatility of $\alpha = 1$ in the original model). Then not all choices of $\tilde \beta \in [0,3]$ are reached by the original model: indeed, one immediately obtains from \eqref{temp933} that $\tilde \beta$ needs to lie in $[1,3]$. 

  In the choice where only one parameter (either $\alpha$ or $\beta$) carries uncertainty, this problem of course vanishes. This is the case for the existing transformations of $g$- and $G-$Brownian motion in the literature and we provide further examples in this direction below.
\end{example}	

The above example also illustrates, that non-linear transformations of processes under ambiguity should be handled with care. 
The following example shows how to obtain a geometric kind of dynamics, which allows us to obtain the \emph{non-linear Black-Scholes} model as considered in \cite[Example 3]{EpsteinJi2013} and \cite{Vorbrink14}.
Both works consider the case where there is only volatility uncertainty.

\begin{example}
	Let $\cA(x,\Theta)$ be a non-linear affine process and consider $F(X)= e^X$. 
  Again, we apply Proposition \ref{Ito}. 
  First, note that with $\tilde x = e^x$,
	$$ a^F(x) = (e^x)^2 a^*(x). $$
	Moreover, since $a^*(x) = [\underline a^{0}+\underline a^{1} x^+, \bar a^{0}+  \bar a^{1} x^+]$, we obtain
	\begin{align}
	   a^F(x) = (\tilde x)^2[\underline a^{0}+\underline a^{1} (\log \tilde x)^+, \bar a^{0}+  \bar a^{1} (\log \tilde x)^+] = \tilde a(\tilde x),
	\end{align}
	and we already computed $\tilde a$. In a similar manner, one obtains $\tilde b$ from \eqref{eq:bF} noting that
	\begin{align*}
	{b}^F(x) &= 
	\begin{cases}
     [e^x( \underline b^0 + \underline b^1 x)+ \frac{1}{2}e^x( \underline a^0 + \underline a^1 x^+), 
	 e^x ( \bar b^0 + \bar b^1 x)+ \frac{1}{2}e^x ( \bar a^0 + \bar a^1 x^+)], & x  \ge 0 \\
	 [e^x( \underline b^0 + \bar b^1 x)+ \frac{1}{2}e^x \underline a^0 , 
	 e^x ( \bar b^0 + \underline b^1 x)+ \frac{1}{2}e^x  \bar a^0 ],& x < 0.
	\end{cases} 
	\end{align*}
     The state space of $e^X$ is of course $\R_{>0}$.
\end{example}

\begin{example}[The non-linear Black-Scholes model]
Allowing for \emph{drift} and \emph{volatility} uncertainty in the log-price of a stock, one arrives
at a non-linear Black-Scholes model. We consider a Brownian motion with drift and volatility uncertainty, which is in our language a non-linear Vasi\v cek model with $\underline b^1=\bar b^1 = 0$. Furthermore, we assume that 
the stock price is given by  $S = \exp (X)$, i.e.~$F(x)=e^x$. Then, the calculations from the previous example 
immediately yield that the stock price is given by the non-linear process $\tilde X$ where 
$a^F(x)=\tilde a(F(x))$ and $b^F(x)=\tilde b(F(x))$ with 
$$ \tilde a(x) = x^2 [ \underline a^0, \bar a^0] $$
and 
$$ \tilde b(x) =  [ x \underline b^0 + \tfrac 1 2 x \underline a^0, x \bar b^0 + \tfrac 1 2 x \bar a^0]. $$
Option pricing for monotone convex (concave) pay-offs can immediately be done by Proposition \ref{prop:icx}, see  Example 3 in \cite{EpsteinJi2013} for explicit formulae for call options (with no uncertainty of the drift). The article \cite{Vorbrink14} excludes drift uncertainty by arguing that under risk-neutral pricing the drift is known. 
\end{example}

\begin{example}[The Heston model with uncertainty in the volatility parameters]\label{ex:Heston}~\\
The model put forward in \cite{Heston} is one of the most popular models for stochastic volatility, which also is heavily used in foreign exchange markets. Model and calibration risk is an important issue, see for example \cite{GuillaumeSchoutens2012} in this regard.  
Here we give a short outline how a non-linear version could be constructed, allowing for parameter uncertainty in volatility only (and not in the drift of the stock price or in the correlation of volatility and stock price). In this regard, we extend $\Omega$ in the classical way to construct an additional (independent) Brownian motion $\tilde V$ which allows us to construct two correlated Brownian motions $V$ and $W$. The correlation is fixed and denoted by $\rho$. Each $P\in \fP(\Omega)$ is extended by leaving $\tilde V$ untouched, such that $(V,W)$ will be a two-dimensional Brownian motion where $V$ and $W$ have correlation $\rho$ and we denote this new semimartingale law again by $P$.

Consider a non-linear CIR process $\cA(x,\Theta)$ with state space $\ccO=\R_{>0}$ as introduced in Section \ref{sec:CIR}. The stock price $S$ is given by the strong solution of the SDE
$$
	dS_t = S_t   X_t dV_t , \qquad 0 \le t \le T
$$ where $X$ is the canonical process on $\Omega$ (and hence a non-linear process). Hence the volatility $X$ stems from a non-linear CIR model which means intuitively, that we have a CIR model with parameter uncertainty with upper and lower bounds $\bar b^0, \bar b^1, \bar a^1$ and $\underline b^0, \underline b^1, \underline a^1$, respectively. For simplicity we chose a vanishing risk-free rate of interest.  

We show how to compute a call-price in this non-linear Heston model in the following. The call price $C(T,K)$ for maturity $T\le T^*$ and strike $K>0$ is given by 
the supremum of the expectations $E^P[(S_T-K)^+]$ over all (extended) semimartingale laws $P$ from $\cA(x,\Theta)$. 
Since the pay-off function $(s-K)^+$ is increasing and convex, the arguments of Proposition \ref{prop:icx} apply and 
$$ C(T,K) = E^{\bar P}[(S_T-K^+)], $$
where $\bar P$ is the worst-case semimartingale law which achieves the supremum. Again from the proof of Proposition \ref{prop:icx} we find that under $\bar P$, $X$ is a (classical) CIR-process with parameters $\bar b^0, \bar b^1, \bar a^1$. The call price formula can be found in \cite{Heston}, see also Section 10.3.3.~in \cite{Filipovic2009} for a derivation using Fourier inversion techniques. 
\end{example}

\section{Affine term structure models}\label{sec.:ATSM}
One of the most important application of affine models is in term structure models. In this regard, we provide in the following a term-structure equation for non-linear affine models implying prices for derivatives or bond-prices. 

Consider a payoff $f(X_T)$ taking place at time $T>0$. In the classical setting, arbitrage-free prices are given by expectations of the discounted pay-off under a risk-neutral measure. 
According to the superhedging duality in \cite{Biaginie.a.2017} (Theorem 5.1), in the case we consider here -- when there is a family of such measures  --  upper bounds of these price processes (and hence the smallest superhedging price) given $X_t=x$ are given by
\begin{align}\label{eq:FTt}
F(T-t,x) = \sup_{P \in \cA(t,x,\Theta)}E^P\Big[ e^{-\int_t^T X_s ds} f(X_T)| X_t=x\Big], \quad 0 \le t \le T.
\end{align}
The following result states the non-linear term-structure equation for the pay-off $f(X_T)$.

\begin{proposition}
Assume that $f$ is Lipschitz-continuous. Then $F(t,x)$ is a viscosity solution of
\begin{align}\label{eq:F}
\partial_t F(t,x) - \sup_{\theta \in \theta } \ccL^\theta F(t,x) + x F(t,x) &=0,
\end{align}
with boundary condition $F(0,x)=f(x)$. If in addition,
\begin{enumerate}[(i)]
\item $\underline a^0 >0$, $\ccO=\R$, and $f$ is bounded, then $F(t,x)$ is the unique solution of \eqref{eq:F}, or
\item if $\underline a^0=\bar a^0=0$, $\underline b^0 \ge \nicefrac{\bar a^1} 2>0$ and $\ccO=\R_{>0}$, then
$F(t,x)$ is the only viscosity solution, such that
	\begin{align}
	\sup_{(t,x) \in [0,T]\times \R_{>0}} \frac{|F(t,x)|}{1+x} < \infty.
	\end{align}
\end{enumerate}
\end{proposition}

For a proof of this result one can argue the same way as in the proof of Theorem \ref{thm:PDE}. More precisely, dynamic programming yields for any stopping time $\tau$ taking values in $[t,T]$ that
$$ F(T-t,x) = \sup_{P \in \cA(t,x,\Theta)}E^P\bigg[ e^{-\int_t^\tau X_s ds} F(T-\tau,X_\tau)\bigg]. $$
Then, following the arguments in Theorem \ref{thm:PDE} leads to the desired result.
Alternatively, one could also enlarge the state space to transform this control problem which is in Lagrange form to one of the Mayer form like in Proposition \ref{prop:DPP} and Theorem \ref{thm:PDE} (see for example Remark 3.10 in \cite{BouchardTouzi2011}).

The term-structure equation now allows to obtain the bond prices %
by considering the pay-off $f(X_T)=1$.  We illustrate how an extension of the state space can be used to achieve a result similar to Proposition \ref{prop:icx} leading to closed-form bond prices in special cases.

In our approach,  upper  bond prices under the non-linear affine term structure model $\cA(t,x,\Theta)$, $x \in \ccO$, are given by
\begin{align}
  \bar p(t,T,x) = \sup_{P \in \cA(t,x,\Theta)}E^P\Big[ e^{-\int_t^T X_s ds} | X_t=x\Big], \quad 0 \le t \le T,
\end{align}
conditional on $X_t=x$. For arbitrary $\omega \in \Omega$ one obtains the bond price as $\bar p(t,T)(\omega) = \bar p(t,T,\omega_t)$.
 The respective lower bond price $\underline p(t,T,x)$ is obtained by replacing the supremum with an infimum. The following proposition shows that in important special cases these prices can be obtained in closed form.
 Again, for  $x \in \ccO$ recall from \eqref{astar} and \eqref{bstar}
 the upper bound of  $a^*(x)$ and $b^*(x)$ defined by
 \begin{align*}%
 \bar a(x)  :=   \bar a^0 + \bar a^1 x^+, 
 \quad \quad
  \bar b(x)  :=  \bar b^{0} + \underline b^1 x \ind{x <0} + \bar b^1 x \ind{x \ge 0}, 
 \end{align*}
and recall $ \bar B^{1,x} =  \underline b^1  \ind{x <0} + \bar b^1  \ind{x \ge 0}$,  so we obtain that 
 $\bar b(x) = \bar b^{0} + \bar B^{1,x} x$.
 Moreover, recall from \eqref{SDEAB} the affine process with  coefficients $\bar a^{0}, \bar a^{1}, \bar b^{0}, \bar B^{1,x}$. 
\begin{proposition}\label{prop:bondprices}
Consider a non-linear affine process $\cA(x,\Theta)$, assume  for all $P \in \cA(x,\Theta)$ that
	\begin{align}
	   \beta^P_t \le \bar{b}^0 + \bar B^{1,x} X_t
	\end{align}
	$dP\otimes dt$-almost everywhere for $0 \le t \le T$ 
	and assume that either $\underline a^1 = \bar a^1 = 0$ or that for all $P\in\cA(x,\Theta)$, $X_t \ge 0$ $P \otimes dt$-a.s.
	If there exists  $\bar P\in\cA(x,\Theta)$ and  a $\bar P$-$\bbF$-Brownian motion $W$ such that the canonical process under $\bar P$ is the unique strong solution of \eqref{SDEAB}, then, for all $u\ge 0$ and $0 \le t \le T$,
	$$ \bar p(t,T,x) = \exp\big( \phi(t,u) + \psi(t,u) x \big), $$
	where $\phi,\ \psi$ solve the Riccati equations
 \begin{align} \label{Ric1bond}
    \partial_t \phi(t,u) &=  \frac{1}{2} \bar a^0 \psi(t,u)^2 + \bar  b^0 \psi(t,u) ,  \phantom{-1} \ \,\qquad \phi(0,u)= 0,\\
    \partial_t \psi(t,u) &=  \frac{1}{2} \bar a^1 \psi(t,u)^2 + \bar  B^{1,x} \psi(t,u)-1, \qquad \psi(0,u)= u.
    \label{Ric2bond}
 \end{align}
\end{proposition}

\begin{proof}
This results also follows by using semimartingale comparison. In this regard let $P \in \cA(x,\Theta)$ and consider the two-dimensional process $Y=(Y^1,Y^2)$ where $Y^1 = -\int_0^\cdot X_s ds$ and $Y^2 = X$. Then, there is no parameter uncertainty with respect to the dynamics of $Y^1$ since its differential semimartingale characteristics are obtained from $dY^1_t =- X_t dt$. 

 By assumption there exists a $\bar P\in \cA(x,\Theta)$ and a $\bar P$-$\bbF$-Brownian motion $W$, such that $X$ is the unique strong solution of the  stochastic differential equation  \eqref{SDEAB}.
 Denote by $\beta, \alpha$ the differential semimartingale characteristics under $P$ of the two-dimensional semimartingale $Y$ and by $\bar \beta, \bar \alpha$ those of $Y$ under $\bar P$. It is easily verified that $\beta_t \le \bar \beta_t$ and also that $\alpha_t \le_{\text{psd}} \bar \alpha_t$ in the positive semidefinite order\footnote{ Here, $A \le_{\text{psd}} B$ means $x^\top (B-A) x \ge 0$ for all $x \in \R^2$.}.

Since $(y^1,y^2) \mapsto e^{u^1 y^1}$ for $u \ge 0$ is increasing and convex, Theorem 2.2 in \cite{BergenthumRueschendorf2007}  (the propagation of order (PO) property is shown in Appendix \ref{app:PO}) yields that,
$$ E^P\big[e^{-\int_0^t X_s ds}] \le  E^{\bar P}\big[ e^{-\int_0^t X_s ds}\big]. $$
Since, $P\in\cA(x,\Theta)$ was chosen arbitrarily and under $\bar P$, $X$ is an affine process with parameters 
$\bar a^{0}, \bar a^{1}, \bar b^{0}, \bar B^{1,x}$, the affine representation follows and the Riccati equations in \eqref{Ric1bond}-\eqref{Ric2bond} can be obtained directly from  Theorem 10.4 in \cite{Filipovic2009}.
\end{proof}

\begin{remark}Again, as in Remark \ref{remarktoprop:icx}, it can be easily checked that the conditions for Proposition \ref{prop:bondprices} hold in the following two cases:
\begin{enumerate}
\item The non-linear Vasi\v cek model with state space $\ccO=\R$ and $\bar b^1 = \underline b^1$,
\item the non-linear CIR model on the state space $\ccO=\R_{>0}$.
\end{enumerate}
It is remarkable, that in the general non-linear Vasi\v cek model with parameter uncertainty on the speed of mean reversion, the classical exponential affine bond pricing formula ceases to hold. Thus, in this model, the interval of parameters can not directly be backed out from a interpolation of bid and ask prices with a standard Vasi\v cek model. This however is the case in the CIR model and in the Vasi\v cek model where $b^1$ is known.
\end{remark}

The previous results and examples directly allow the treatment of term-structure models based on the non-linear Vasi\v cek model and on the non-linear CIR model (see Sections \ref{non-linearVasicek} and \ref{sec:CIR}). 
An important difficulty for the modeller in practical situations is that she has to make her choice between these  models with strong implications: for example, the state space can allow for negative values or can strictly exclude them, a well-documented difficulty of the affine models after the European crisis, see \cite{Carver2012}. The following example shows that in a non-linear setting one is able to mix these two type of models and the modeller no longer has to decide a priori if she allows or excludes negative values.

\subsection{The non-linear Vasi\v cek-CIR model} \label{sec:non-linearVasicekCIR}

If one is not able to restrict the state space a priori, one can consider the following non-linear affine model: assume that both parameters $a_0$ and $a_1$ are subject to parameter uncertainty (or at least one of them with the other parameter not vanishing). Intuitively, this means that the model may switch between a Vasi\v cek-like or CIR-like behaviour. In particular, when one is not able to restrict the state space a priori to $\R_{\ge 0}$, this non-linear model allows to incorporate both model approaches in a robust (i.e.~non-linear) sense.

In the interest rate markets in the early years after 2000, market participants believed in positive interest rates and thus favoured the CIR-model. The credit crises led to decreasing interest rates and the Vasi\v cek model came back, as it allows for negative interest rates. This effect is well-known and its implications for banking are quite important, see for example \cite{Carver2012,Patel2017,orlando2016,RussoFabozzi2017}.
In the near future, however, when interest rates may rise again one could be interested in deviating again from the Vasi\v cek-model. With the non-linear Vasi\v cek-CIR model such a switch is no longer necessary and one is able to behave consistently through such seemingly different time periods. 

More precisely, assume that $\underline a^0 >0$ and consider the state space $\ccO=\R$. Uniqueness for the non-linear PDE \eqref{eq:def:PDE} follows readily from Proposition \ref{thm:PDE-uniquenessI}. In this general case, there will be no explicit solutions like for example in Proposition \ref{prop:bondprices} above and we have to rely on numerical techniques.

Figure \ref{fig:VasicekCIR} illustrates the differences in option pricing of the non-linear Vasi\v cek-CIR model to both the Vasi\v cek and the CIR model. While in the non-linear CIR model, this price can still be computed explicitely, this no longer holds for the other two models.  
The parameters of the plot are chosen for best illustration purposes. 
The plot shows the (upper) price of a Call option under a non-linear affine model with the parameters stated in the plot. 
While the non-linear Vasi\v cek model puts clearly weight on $\R_{<0}$, the CIR model shifts the mass towards the right, leading to higher option values for $x > 0.3$. 

The graph further highlights the mentioned difficulty of the modeler to choose from one model: starting from positive values (i.e.~here $x> 0.3$), he possibly would prefer to choose the (classical/non-linear) CIR model, while for falling values prices become prohibitively low, illustrating the need to switch to the Vasi\v cek model. The non-linear Vasi\v cek-CIR model clearly does not suffer from this problem (it is clear by definition that the parameters of the non-linear Vasi\v cek-CIR model can be chosen to reproduce the option prices of the two other models).

\begin{figure}[h]

	 \begin{overpic}[width=12cm, trim=0 0 0 30, clip]{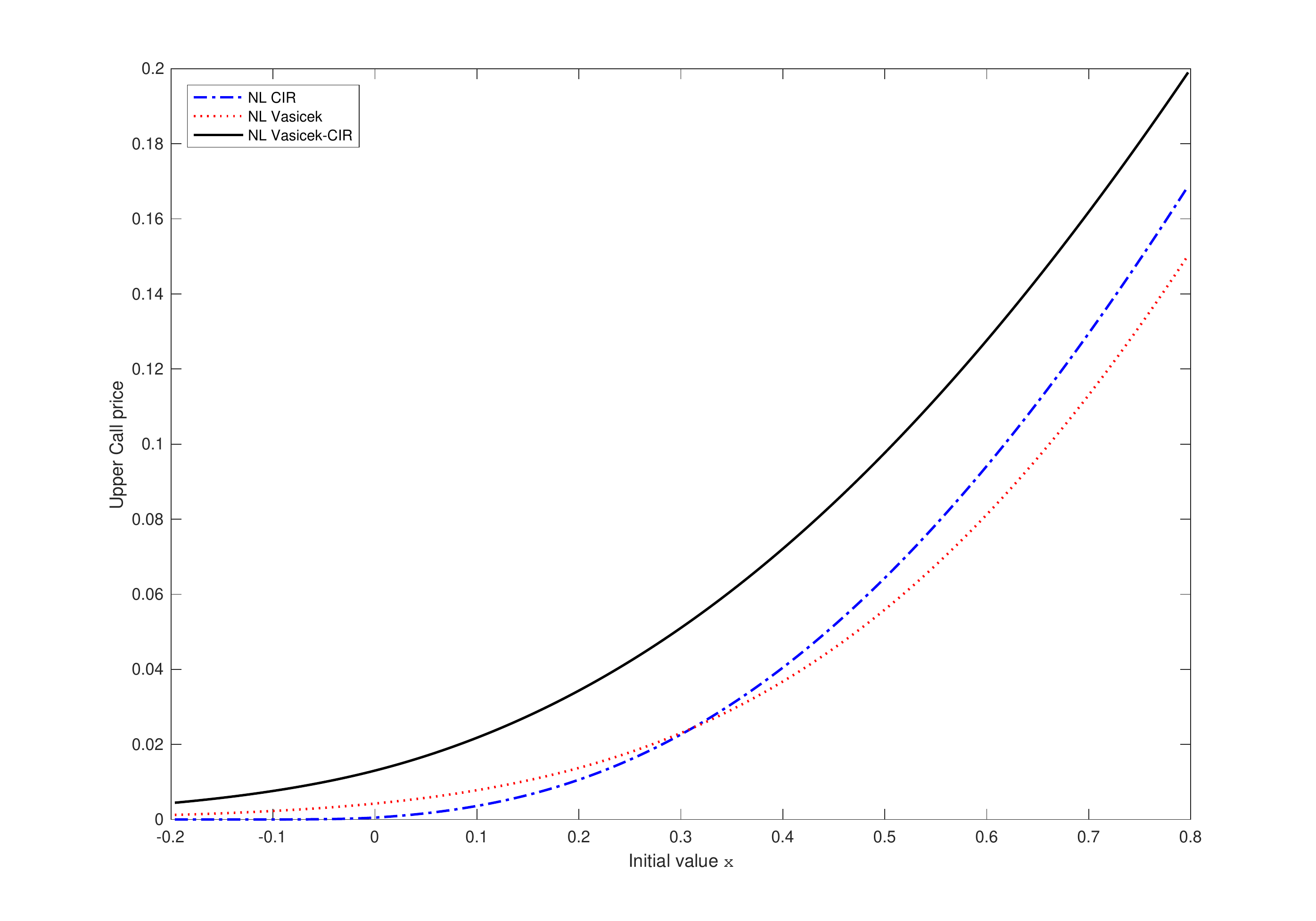} \end{overpic}  
	 .

\begin{tabular}{lrrrrrrrr} \toprule
Parameters & $\overline a_0$ & $\underline a_0$ & $\overline a_1$ &  $\underline a_1$ & $\overline b_0$  & $\underline b_0$ & $\overline b_1$ & $\underline b_1$ \\ \midrule
& 0.08 & 0.00 & 0.2 & 0.00 & 0.15 & 0.05 & -0.5  & \ -1  \\ \bottomrule
\end{tabular}

\caption{\label{fig:VasicekCIR} This figure shows the upper Call price for the non-linear Vasi\v cek, the non-linear CIR and the non-linear Vasi\v cek-CIR model. The first two models are obtained from the latter by simply letting $\bar a^1=0$ ($\bar a^0 = 0$, respectively). The Call price has strike $0.5$ and the parameters are given in the table above. }
\end{figure}

\section{Model risk}\label{sec:modelrisk}

In financial applications, model risk is an important factor for risk management. In the remarkable work \cite{cont-06}, a systematic framework for the management of model risk has been proposed which we recall shortly and thereafter apply to the non-linear affine models. The importance of this topic is illustrated by the intensive research in this area, see for example \cite{BannoerScherer13,GuillaumeSchoutens2012,BreuerCsiszar2016,BarrieuScandolo2015,FonsecaGrasselli2011} among many others.

In this approach, the market contains a number of benchmark instruments which are liquidly traded instruments and the observation consists in bid and ask prices thereof. Moreover, there is a set of arbitrage-free pricing models $\cQ$ which is consistent with the observations of the benchmark instruments.

In our framework, both can be described through a non-linear affine model: the non-linear affine model specifies a set $\cQ$ of pricing measures, as for example in the non-linear affine term-structure approach studied in Section \ref{sec.:ATSM}. Consistent bid and ask prices can be obtained by suprema and infima over these pricing measures, 
exactly as it was done for $\underline p(t,T)$ and $\bar p(t,T)$ in the previous section. 

A \emph{coherent} measure of model uncertainty for a payoff function $\psi:\R \to \R$ with $\psi(X_T)$ denoting the payoff at time $T$ can be computed from the upper and lower price bounds
\begin{align}
    \bar \pi(\psi) = \sup_{Q \in \cQ}E^Q[\psi(X_T)], \qquad \underline \pi(\psi) = \inf_{Q \in \cQ}E^Q[\psi(X_T)].
\end{align}
The measure of model uncertainty on the derivative $\psi$ is given by
$$\mu_\cQ(\psi) = \bar \pi(\psi) - \underline \pi(\psi). $$

Some examples are provided in \cite{cont-06}, including the non-linear Black-Scholes model. 
We illustrate the application of non-linear affine models in this framework with a short study of model risk in the non-linear Vasi\v cek model.

\subsection{The non-linear Vasi\v cek model}
As an example we consider the non-linear Vasi\v cek-model introduced in Section \ref{non-linearVasicek}.
 Recall that this model is characterized by the assumption that $[\underline a^1,\bar a^1]=\{0\}$. 
In the non-linear case, non-linear expectations are given by the solution of the non-linear Kolmogorov equation \eqref{eq:Kolmogorov}. Proposition \ref{prop:icx} allows to trace this solution back to existing solutions for affine models if the payoffs are increasing and convex (decreasing and concave), see Example \ref{ex:Heston}. For more general payoffs we rely on numerical methods which we illustrate now.

\subsubsection{Options} 
The model risk for options in this model is illustrated in the following two pictures, where we price a call and a butterfly. To construct the set $\Theta$, we take estimated parameter values together with their 95\% confidence intervals from the literature. The results is shown in Figure \ref{fig:Vasicekder}. While for the Call option the model risk increases monotonically with the initial value, the maximal model risk for the butterfly is attained for the initial value $x$ directly at the maximal payoff.

\begin{figure}[h]

\begin{minipage}{20cm}
\hspace{-1.1cm}
	 \begin{overpic}[width=7.5cm]{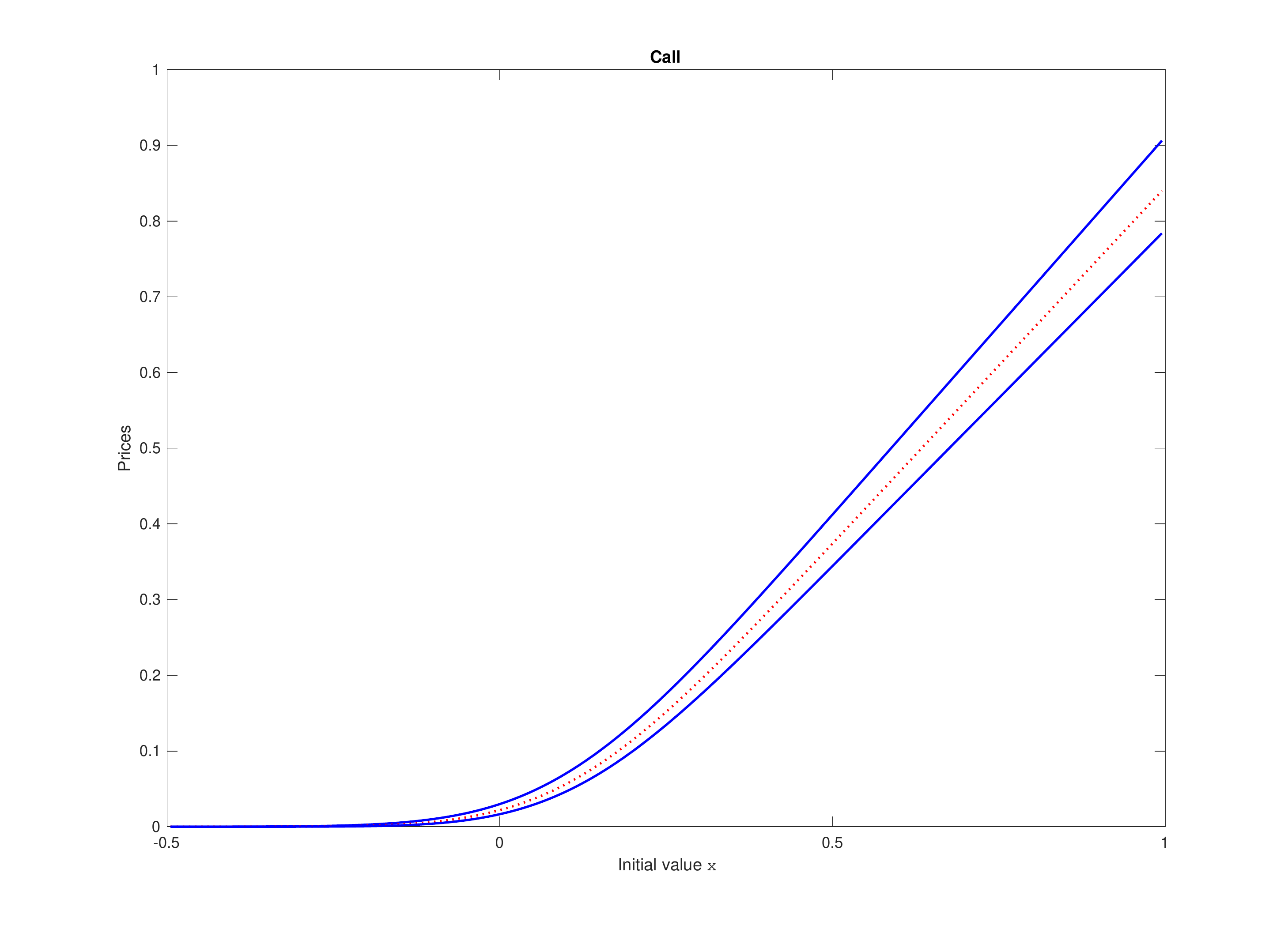} \end{overpic}  \hspace{-8mm}
	 \begin{overpic}[width=7.5cm]{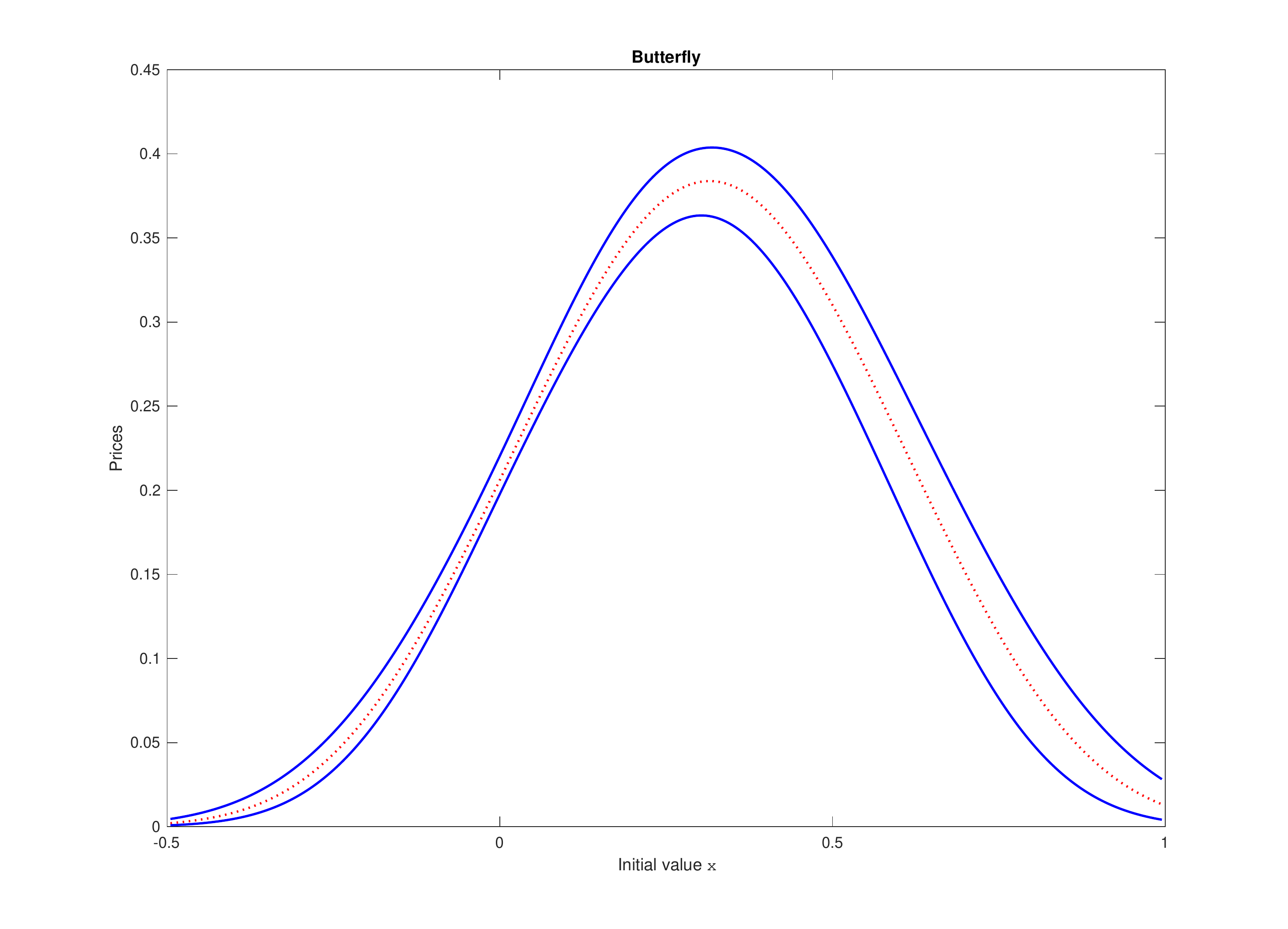} \end{overpic} 
\end{minipage}
\begin{tabular}{lrrrrrrrrr} \toprule
Parameters & $\overline a_0$ & $a_0$ & $\underline a_0$ & $\overline a_1$ & $a_1$ & $\underline a_1$ & $\overline b_0$ & $b_0$ & $\underline b_0$ \\\midrule
& 0.0003 & 0.003 & 0.017 & 0.00 & -0.06 & -0.11 & 0.026 & 0.023 & 0.019 \\ \bottomrule
\end{tabular}
\caption{\label{fig:Vasicekder} This figure shows the solution of the non-linear Kolmogorov equation for the non-linear Vasi\v cek-model with boundary condition $f(x)=(x-K)^+$ and $f(x)=  (x-K_1)^+ -2(x-K_2)^+ + (x-K_3)^+$, $K=0.1$, $K_1=-0.2$, $K_2=0.3$, and $K_3=0.8$. 
For each of the figures, the dashed lines show the solution of the Vasi\v cek model with no uncertainty. 
The upper and lower solid lines gives the upper and lower price bounds. The parameters for the computation are given in the table below the pictures.}
\end{figure}

\section{Conclusion}\label{sec:conclusion}
In this paper we introduced affine processes under uncertainty. This extends the existing class of non-linear L\'evy processes  to Markov processes where the interval for the parameter uncertainty may depend on the current state (in an affine way, however). We obtained a dynamic programming principle implying a non-linear Kolmogorov equation which can be used to price options in a fast and efficient way. Many existing models can be embedded into a setting with parameter uncertainty which we illustrate with a number of examples. However, the non-linear framework also allows for new model variations which did not exist in the classical approach and we illustrate this with a term-structure Vasi\v cek-CIR model, where the modeler does not need to  decide a priori if the state space should include negative rates or not, a strong restriction in existing models. 
The generalization to higher dimensions or to the case with jumps is left for future research. Here, we  concentrated on the conceptual introduction of state-dependent parameter uncertainty and chose the simplest but still highly interesting example for the illustration of our ideas.

\begin{appendix}
\section{Semimartingale comparison} \label{app:PO}
In this section we recall the main result in \cite{BergenthumRueschendorf2007} on comparison of semimartingales with Markov processes and show that the crucial \emph{propagation of order} (PO) property is satisfied for a large class of Markov processes, in particular for affine processes. To the best of our knowledge, existing results in the literature require Lipschitz assumptions on the coefficient, which will not hold in our case. 

As function class $\cF$ we will consider increasing and convex functions. For a real-valued Markov process $S^*$ and a terminal time $T$ we define the propagation operator
$$ \cG^g(t,x) = E[g(S^*_{T})|S^*_t=x]. $$

\begin{assumption}For some function class $\cF$ and some Markov process $S^*$ we say that PO$(S^*,\cF)$ holds if $\cG^g(t,\cdot) \in \cF$ for all $0 \le t \le T$ and for all $g \in \cF$.
\end{assumption}

Propagation of monotonicity and convexity follows in a very elegant way through total positivity of the transition densities of continuous Markov processes.   

Proposition 3.1 in \cite{Kijima2002} immediately yields the following result. 
\begin{proposition}\label{prop:i}
Assume that $S^*$ is a strong Markov process having continuous sample paths and that 
 $g$ is increasing (decreasing). Then $\cG^g(0,\cdot)$ is increasing (decreasing).
\end{proposition}

For the propagation of convexity we need an additional step, because the considered processes in \cite{Kijima2002} are in fact martingales. The proof crucially uses the variation-diminishing property of totally positive functions.
\begin{proposition}\label{prop:cx}
Assume that $S^*$ is a strong Markov process with state space $S$ having continuous sample paths and that there exist $\pi_0,\ \pi_1 \in \R$ with $\pi_1 \neq 0$, such that 
\begin{align}\label{pi}
          E[S^*_{T}|S^*_0=x]=\pi_0 +  x \pi_1, \qquad x \in S.
\end{align}
Then for convex (concave) functions  $g$  it holds that $\cG^g(0,\cdot)$ is convex (concave).
\end{proposition}
\begin{proof}
We modify the first step in Proposition 3.2 in \cite{Kijima2002}. In this regard, note that Equation \eqref{pi} yields that
	\begin{align}\label{eq:x}
		x &= \frac{E[S^*_{T}|S^*_0=x] - \pi_0}{\pi_1}. 
	\end{align}
Moreover, we follow the notation in \cite{Kijima2002} and denote by $q_T(x,y)$ the transition density of the Markov process $S^*$, i.e.
	\begin{align}\label{eq:cG}
		\cG^g(0,x) = E[g(S^*_{T})|S^*_0=x] = \int_S g(y) \, q_T(x,y) dy. 
	\end{align}
Let $\tilde \alpha_1,~\tilde \alpha_2 \in \R$. From \eqref{eq:x} and \eqref{eq:cG} we obtain that
	\begin{align*} 
		\cG(0,x) - \tilde \alpha_1 x - \tilde \alpha_2 & =\int_\R q_{T}(x,y) \Big\{ g(y)- \tilde \alpha_1 x  - \tilde \alpha_2\Big\} dy \\
		&= \int_\R q_{T}(x,y) \Big\{ g(y)- \frac{\tilde \alpha_1}{\pi_1} y  + \frac{\tilde \alpha_1\pi_0}{\pi_1} -\tilde \alpha_2\Big\} dy. 
	\end{align*}
Since $ \frac{\tilde \alpha_1}{\pi_1} y  + \frac{\tilde \alpha_1\pi_0}{\pi_1} -\tilde \alpha_2=: \alpha_1 y + \alpha_2 $ is an affine function of $y$,  the claim follows by repeating the arguments in Proposition 3.2 in \cite{Kijima2002}.
\end{proof}

Consider in addition a different Markov process $S$, possibly on a different probability space and denote by
\begin{align*} 
	\cF_{icx}:=\{ f:\R \to \R, \text{increasing and convex} \}.
\end{align*}

A combination of Proposition \ref{prop:i} with \ref{prop:cx} yields the following result. 
\begin{proposition}\label{prop:PO}
Let $S^*$ be a strong and homogeneous Markov process with continuous sample paths and that    
\begin{align}
          E[S^*_{t}|S^*_0=x]=\pi_0(t) +  x \pi_1(t), \qquad 0 \le t \le T
\end{align} holds with $\pi_1(t)\neq 0$ for all $t \in [0,T]$.
Then PO$(S^*,\cF_{icx})$ holds.
\end{proposition}
\begin{proof}
Let $g \in \cF_{icx}$. 
First, Proposition \ref{prop:i} yields that $\cG^g(0,\cdot)$ is increasing. As the choice of $T$ was arbitrary, we obtain that also $\cG^g (t,\cdot)$ is increasing by repeating the argument of Proposition \ref{prop:i} and using homogeneity of $S^*$. 

Second, Proposition \ref{prop:cx} yields that $\cG^g(0,\cdot)$ is convex. Denote 
$$ \cH^g(t,x) = E[g(S^*_{t})|S^*_0=x]. $$
Then, $\cG^g(t,x)=\cH^g(T-t,x)$ since the Markov process is homogeneous. Now Proposition \ref{prop:cx} yields that $\cH^g(t,\cdot)$ is convex since the choice of $T$ was arbitrary and the claim follows.
\end{proof}

We remark that for a continuous, one-dimensional affine process $X$ 
with $[\underline a^1,\bar a^1]=\{0\}$ or  $\ccO=\R_{>0}$, affinity of the expectation in \eqref{pi} is satisfied which follows 
directly from the exponential-affine structure of the Laplace transform $ E[e^{uX_t}|X_0=x]=\exp(\phi(t,u)+\psi(t,u) x)$. Indeed, an application of Theorem 13.2 in \cite{JacodProtter2004} yields that
$$ E[X_t|X_0=x] = \partial_u \phi(t,0)+ x \partial_u \psi(t,0), $$
where we use the explicit expressions for $\phi$ and $\psi$ from Section 10.3.2 in  \cite{Filipovic2009}.

\section{Comparison results} \label{app:COMP}
In this section we recall the comparison results from \cite{Amadori2007} in our notation. Again, the crucial point for this results is that Lipschitz assumptions on the full domain do not hold. Note that we only consider the one-dimensional, time-homogeneous case here, which simplifies the matter significantly. While minimization is the core topic of \cite{Amadori2007}, the financial applications mainly treat maximization, such that we concentrate on the maximization. The stated results follow from the original results by replacing $\psi$ with $-\psi$.

Fix the state space $\ccO=\R_{>0}$ and consider the controlled diffusion $X=X^\theta$
   \begin{align}\label{app:controlleddiffusion} 
      dX_s = b(X_s,\theta_s) ds + \sqrt{a(X_s,\theta_s)} dW_s, \quad s>t 
   \end{align}
with initial condition $X_t=x \in \ccO$. 
Our application will be in Proposition \ref{thm:PDE-uniquenessII}, which considers the non-linear CIR-modell. Since then $\underline a^0=\bar a^0=0$ we consider $\Theta = [\underline b^{0},\bar b^{0}] \times [\underline b^{1},\bar b^{1}] \times \{0\} \times [\underline a^{1}, \bar a^{1}] $ and an adapted process $(\theta_s)$ taking values in $\Theta$. Moreover,  for each $x\in \ccO$, $\theta\in \Theta$, $b(x,\theta)=b^0+b^1x$ and $\sqrt{a(x,\theta)}=\sqrt{a^1 x}$ which is not Lipschitz at $0$.

Now, fix a continuous function $\psi:\ccO \to \R$ and  define the value function $v:[0,T]\times \ccO\to \R$ by
\begin{equation*}
v(t,x):=\sup_{(\theta_s)} E_{x,t}[\psi (X_T)],
\end{equation*}
where the supremum ranges over all adapted processes $(\theta_s)$ taking values in $\Theta$, $X$ is the controlled diffusion satisfying \eqref{app:controlleddiffusion},  and $E_{t,x}$ refers to the conditional expectation conditioning on $X_t=x$.
The associated \emph{Hamilton-Jacobi-Bellman} equation is given in \eqref{eq:def:PDE}.

First, Assumption 1 in \cite{Amadori2007} holds. Indeed, note that the functions $f$ and $r$ therein equal to zero in our case, that the functions $b$ and $\sqrt{a}$ are Lipschitz-continuous on $[\epsilon,\infty)$ for all $\epsilon>0$ and all $\theta \in \Theta$. Moreover, let $\parallel \cdot \parallel_{C^{0,1}([\epsilon,\infty))}$ denote the Lipschitz\footnote{This is the H\"older coefficient for exponent $\alpha$  with $\alpha=1$, see Section 5.1 in \cite{evans10}.} coefficient on $[\epsilon,\infty)$, then clearly
$$ \sup \{ \parallel \sqrt{a(\cdot,\theta)} \parallel_{C^{0,1}([\epsilon,\infty))}: \theta \in \Theta \} < \infty,$$
and, similarily, for $b$ and all conditions of Assumption 1 hold. 

Second, Assumption 4 is  implied by the Feller condition $\underline b^0 \ge \nicefrac{\bar a^1} 2>0$. Indeed, Assumption 4 requires that 
$$ \limsup_{x \to 0} \sup_{\theta \in \Theta}\Big( \frac 1 x - \frac{2b(x,\theta)}{a(x,\theta)} \Big) < \infty. $$
Since $\underline b^0 \ge \nicefrac{\bar a^1} 2> 0$ implies that
\begin{align*}
   \frac 1 x - \frac{2b(x,\theta)}{a(x,\theta)} &= \frac 1 x \Big(1- \frac{2b^0}{a^1  }\Big) -\frac{2 b^1}{a^1}
       \le -\frac{2 b^1}{a^1},
\end{align*}
Assumption 4 also holds.
In our notation, the uniqueness result following immediately from the comparison principle given in Theorem 4 in \cite{Amadori2007} reads as follows. We refer to the Section \ref{sec:Kolmogorov} for the definitions of viscosity solutions, super- and subsolutions.

\begin{theorem}\label{thm:Amadori}
Assume that $\underline b^0 \ge \nicefrac{\bar a^1} 2>0$. Let $v$  be a continuous locally bounded viscosity solution  of \eqref{eq:def:PDE}.
\begin{enumerate}[(i)]
   \item Assume
   \begin{align} \label{amadoriI} 
      \sup_{\ccO \times [0,T]} \frac{|v(t,x)|}{1 + x} < \infty.
   \end{align}
   Let $u$ be a locally bounded viscosity solution satisfying \eqref{amadoriI}, then $u=v$.
   \item  Assume
   \begin{align}\label{amadoriII}
     & \limsup_{x \to \infty} \sup_{t \in [0,T]} \frac{|v(t,x)|}{x^2} = 0, \quad
        \lim_{x \to 0} \sup_{t \in [0,T]} \frac{|v(t,x)|}{h_\epsilon(x)} = 0.
  \end{align}
  Let $u$ be a locally bounded viscosity solution satisfying \eqref{amadoriII}, then $u=v$.
\end{enumerate}

\end{theorem}
\begin{proof}
Recall that a viscosity solution is a supersolution and a subsolution. Since $u$ and $v$ are both viscosity solutions, $u=v$ holds on $\ccO\times\{T\}$.
For (i), note that the conditions hold with $\gamma=1$ both for $u$ and $v$, such that applying Theorem 4 in \cite{Amadori2007} twice (once as a supersolution and once as a subsolution) yields  $u=v$ on $\ccO \times [0,T]$.
The claim (ii) follows similarly. 
\end{proof}

\end{appendix}

\end{document}